\documentclass[11pt]{article} 

\usepackage{amsmath, amsfonts, amssymb}
\usepackage{mathrsfs}
\usepackage{theorem}
\usepackage{bm}
\usepackage[dvipdf]{graphicx}

\RequirePackage[colorlinks,citecolor=blue,urlcolor=blue,bookmarksopen]{hyperref}

\newtheorem{mythm}{Theorem}[section]
\newtheorem{myprop}[mythm]{Proposition}
\newtheorem{mylem}[mythm]{Lemma}

\newtheorem{mydefn}[mythm]{Definition}

{\newtheorem{myrem}[mythm]{Remark}}
{
}%

\newcommand{\dis}{\displaystyle}

\def\R{\mathbb R}
\def\Z{\mathbb Z}
\def\N{\mathbb N}
\def\C{\mathscr C}
\def\B{\mathscr B}

\def\F{\mathscr F}

\def\E{\mathbb E}
\def\p{\mathbb P}
\def\e{\text{\rm{e}}}

\def\la{\langle}\def\d{\text{\rm{d}}}
\def\raa{\rangle}

\def\veps{\varepsilon}

\def\law{\mathscr{L}}

\def\pb{\mathscr{P}}

\def\wt{\widetilde}

\def\W{\mathbb{W}}
\def\oo{\mathscr{O}}
\def\balpha{\bm \alpha}

\newenvironment{proof}{{\noindent\it Proof.}\ }{\hfill $\square$\par}

\numberwithin{equation}{section}

\allowdisplaybreaks

\begin{document}

\title{Optimal control problem for reflected McKean-Vlasov SDEs\footnote{Supported in part by National Key R\&D Program of China (No. 2022YFA1000033) and NNSFs of China (No. 12271397,  11831014)}}

\author{ Jinghai Shao\\[0.2cm]
Center for Applied Mathematics, Tianjin University, Tianjin 300072, China.
}
\maketitle

\begin{abstract}
  This work investigates the optimal control problem for reflected McKean-Vlasov SDEs and the   viscosity solutions to Hamilton-Jacobi-Bellman(HJB) equations on the Wasserstein space in terms of intrinsic derivative.  It follows from the flow property of reflected McKean-Vlasov SDEs that the dynamic  programming principle holds. Applying the decoupling method and the  heat kernel estimates for parabolic equations, we show that the  value function is a viscosity solution to an appropriate HJB equation on the Wasserstein space, where the characterization of absolutely continuous curves on the Wasserstein space by the continuity equations plays an important role. To establish the uniqueness of viscosity solution, we generalize the construction of a distance like function initiated in Burzoni et al.\,(SICON, 2020) to the Wasserstein space over multidimensional space and show its effectiveness to cope with HJB equations in terms of intrinsic derivative on the Wasserstein space.
\end{abstract}

\textbf{AMS MSC 2010}: 60H10, 35Q93, 49L25

\textbf{Key words}: Viscosity solution, Wasserstein space, Reflected McKean-Vlasov, Feedback control

\section{Introduction}

Let $\mathscr{O}$ be a bounded convex domain in $\R^d$ with $C^{1,1}$ boundary. $\bar{\oo}$ denotes the closure of $\oo$ and $\partial \oo$ its boundary. Let $\vec{\mathbf{n}}(\cdot)$ be the unit outward normal of $\oo$. Consider the following reflected McKean-Vlasov equation
\begin{equation}\label{a-1}
\left\{ \begin{array}{l}
\d X_t=b(t,X_t,\law_{X_t},\alpha_t)\d t+\sigma \d B_t-\vec{\mathbf{n}}(X_t)\d k_t,\\
 k_t=\int_0^t\mathbf1_{\partial \oo}(X_s)\d k_s,
\end{array}\right.
\end{equation}
where $\law_{X_t}$ denotes the distribution of $X_t$.   $(B_t)$ is a $d$-dimension Brownian motion. The coefficients $b,\,\sigma$ will be detailed later. The term $\alpha_t$ represents the control strategy imposed on $(X_t)$. The solution of \eqref{a-1} is a pair $(X_t,k_t)$, and the process $(X_t)$ will stay always in $\bar\oo$. $(k_t)$ is called the local time of $(X_t)$ on $\partial \oo$, which is a continuous process and increases only when $X_t$ hits the boundary $\partial \oo$.

In this  work, we shall investigate the finite horizon optimal control problem for the reflected process $(X_t)$. The associated value function will be defined on the Wasserstein space $\pb(\bar\oo)$, the space of all probability measures over $\bar\oo$. The admissible controls considered in this work are of feedback control form, which contain the set of deterministic controls used such as in \cite{BIRS}. This adds new difficulty to verify the value function to be a viscosity supersolution. The dynamic programming principle is established following from flow property of the solutions to reflected McKean-Vlasov SDEs.   Then, using the intrinsically differential structure for functions on $\pb(\bar\oo)$, and taking advantage of the characterization of absolutely continuous curves on $\pb(\bar\oo)$ in terms of continuity equations (cf. \cite{Amb}),   the value function is proved to be a viscosity solution to an appropriate  HJB  equation on $\pb(\bar\oo)$. In the argument, we use the decoupling method and  the heat kernel estimates for parabolic equations to overcome the difficulty caused by the feedback controls. The intrinsically differential structure on $\pb(\bar\oo)$ is closely related to the study of optimal transport map problem (cf. \cite{Amb,Vil}), and Monge-Amp\`ere equations (cf. \cite{CM10,TW08,Vil}), which can provide rich geometric structure on $\pb(\bar \oo)$.

As being well known in the study of HJB equations over the Wasserstein space, one of the main challenges is to prove the uniqueness of the viscosity solution.
The theory of viscosity solutions for HJB equations in infinite dimensional space was initiated by Crandall and Lions \cite{CL85} on Hilbert space or certain Banach space. However, the recent study on the optimal control problem for McKean-Vlasov SDEs and mean field games has motivated a lot of research interest on HJB equations on the Wasserstein space; cf. e.g. \cite{BIRS,Car1,Car2,LL07,Pham1,Pham2}.

The Wasserstein space has various kinds of differential structure, and correspondingly various HJB equations have been established on it. For example, Ambrosio and Feng \cite{AF14}, Gangbo and Swiech \cite{GS15} used metric derivative to study viscosity solutions of (first order) Hamilton-Jacobi equations on the Wasserstein space. Lions \cite{Lions} lifted functions defined on the Wasserstein space to functions on an appropriate $L^2$ space, and used the well developed viscosity solution theory for HJB equations on Hilbert space to study HJB equations on the Wasserstein space.  Gangbo, Nguyen, Tudorascu \cite{GNT}, Gangbo and Tudorascu \cite{GT19} exploited the isometry between a quotient space  of $L^2$ space to the Wasserstein space at length, and made inferences on partial differential equations in the latter space. Pham and Wei \cite{Pham1,Pham2} studied HJB equations on the Wasserstein space by using Lions' lifting to solve the optimal control problem for McKean-Vlasov SDEs(with common noise). Burzoni et al. \cite{BIRS} investigated the viscosity solutions to HJB equations using linear functional derivative on the Wasserstein space. They raised a distance like function on the Wasserstein space over $\R$, whose linear functional derivative can be controlled by itself. The  construction of this distance like function is quite subtle.

Our strategy to establish the uniqueness of viscosity solution is based on two observations: 1) The $L^2$ Wasserstein distance $\W_2$ is only intrinsically differentiable at the probability measures satisfying certain regular property, and its derivative cannot satisfy the smoothness condition of the HJB equation established by solving the optimal control problem associated with \eqref{a-1}. 2) The distance like function constructed in \cite{BIRS} is also intrinsically differentiable, and its intrinsic derivative is smooth enough to be used to act as the smooth approximation function of the value function.  Thus, we generalize the construction of   \cite{BIRS} for $\pb(\R)$ to the Wasserstein space $\pb(\bar\oo)$ over multidimensional space, and make use of the weak compactness of $\pb(\bar\oo)$ to establish the comparison principle for the viscosity sub/super-solutions to our established HJB equations on $\pb(\bar\oo)$. A technical restriction of this method, like in \cite{BIRS}, is that the drift $b$ can only depend on the finite order moments of $\mu$ and is independent of $x$.

This work is organized as follows. In Section 2, we present the framework of the optimal control problem for \eqref{a-1}, and study the continuity of the value function and establish the dynamic programming principle. In Section 3,  under the intrinsic differential structure of $\pb(\bar\oo)$, the law $\mu_t$ of the controlled process $X_t$ is shown to be an absolutely continuous curve in $\pb(\bar\oo)$, whose velocity $v_t$ can be characterized when $\mu_t$ satisfies certain regular condition. Furthermore, the value function is shown to be a viscosity solution to an appropriate HJB equation on $\pb(\bar\oo)$. In Section 4, we study the regularity of $L^2$ Wasserstein distance $\W_2$ based on the regularity of the solution of Monge-Amp\`ere equation. Then, the uniqueness of viscosity solution is established under the generalization of \cite{BIRS}'s construction of a distance like function.

\section{Framework}

Let $(\Omega, \mathscr{F},\{\F_t\}_{t\geq 0},\p)$ be a complete filtered probability space. Let $\pb(\bar\oo)$ be the space of all probability measures over $\bar\oo$. Let $(\alpha_t)$ be an $\F_t$-adapted process.   Let $U$ be a compact set in $\R^m$ for some $m\in \N$. We shall study the finite horizon optimal control problem, so let  $T>0$ be   given and fixed in this work.
As $\oo$ is bounded, all the probability measures in $\pb(\bar\oo)$ own finite $p$-th moments for all $p\geq 1$. Let
\begin{equation}\label{b-0}
\pb^r(\bar\oo)=\Big\{\mu\in \pb(\bar\oo);\d\mu(x)\ll \d x\ \text{and}\ \rho(\cdot):=\frac{\d\mu }{\d x}\in C^1(\bar\oo),\ \rho(\cdot)>0\Big\}.
\end{equation}

\begin{mydefn}\label{defn-1}
A pair $(X_t,k_t)$ is called a solution to \eqref{a-1}, if $(X_t)$ is an adapted continuous process on $\bar\oo$, $(k_t)$ is an adapted continuous increasing process such that $\p$-a.e.
\[\int_0^t\big(|b(r,X_r,\law_{X_r},\alpha_r)| +\|\sigma\|^2\big)\d r<\infty,\quad t\geq 0,\]
and $\dis  k_t=\int_0^t\mathbf{1}_{\partial \oo}(X_s)\d k_s$,
\[X_t=X_0+\int_0^tb(r,X_r,\law_{X_r},\alpha_r)\d r+\sigma  B_t-\int_0^t \vec{\mathbf{n}}(X_r)\d k_r.\]

A triple $(X_t,k_t,B_t)_{t\geq 0}$ is called a weak solution to \eqref{a-1}, if $(B_t)$ is a $d$-dimension Brownian motion under  a probability space $(\Omega,\F,\{\F_t\}_{t\geq 0},\p)$, $(X_t,k_t)$ solves \eqref{a-1} with initial value $X_0=\xi\in \F_0$:
\begin{gather*}
  X_t=\xi+\int_0^tb(r,X_r,\law_{X_r},\alpha_r)\d r+ \sigma  B_t-\int_0^t \vec{\mathbf{n}}(X_r)\d k_r.
\end{gather*}
If for any two weak solution $(X_t,k_t,B_t)_{t\geq 0}$ under probability $\p$, $(\wt X_t,\tilde k_t,\wt B_t)_{t\geq 0}$ under probability $\wt \p$ satisfying $\law_{X_0|\p}=\law_{\wt X_0|\wt \p}$, then $\law_{(X_t,k_t)|\p}=\law_{(\wt X_t,\tilde k_t)|\wt \p}$ for $t>0$,   SDE \eqref{a-1} is called  weakly unique.

We call \eqref{a-1} weakly wellposed for distributions in $\hat{\pb}$, if it has a unique weakly solution for any $\F_0$-measurable variable $\xi$ with $\law_{\xi}\in \hat{\pb}$, and the distribution of $X_t$ remains in $\hat{\pb}$ for any $t>0$. When $\hat{\pb}=\pb(\bar\oo)$, we simply say that \eqref{a-1} is weakly wellposed.
\end{mydefn}


The $L^p$-Wasserstein distance $\W_p$  for two probability measures $\mu,\nu\in\pb(\bar\oo)$ is defined by
\[\W_p(\mu,\nu)=\inf_{\Gamma\in \C(\mu,\nu)}\Big(\int_{\bar\oo\times \bar\oo} |x-y|^p\Gamma(\d x,\d y)\Big)^{\frac 1p },\quad p\geq 1, \] where $\C(\mu,\nu)$ stands for the collection of all couplings of $\mu$ and $\nu$.

Let $U$ be a compact set in $\R^k$ for some $k\geq 1$. Assume that the coefficients $b:[0,T]\times \bar\oo\times \pb(\bar\oo)\times U\to \R^d$, $\sigma\in \R^{d\times d} $ satisfy:
\begin{itemize}
  \item[$(\mathrm{H}_1)$] $\exists\, K_1>0$ such that for all $s,t\in [0,T]$, $x,y\in \bar\oo$, $\mu,\nu\in \pb(\bar\oo)$, $\alpha,\tilde \alpha\in U$,
  \[|b(t,x,\mu,\alpha)-b(s,y,\nu,\tilde\alpha)| \leq K_1\big(|s-t|+|x-y|+\W_2(\mu,\nu)+|\alpha-\tilde \alpha|\big).\]
  \item[$(\mathrm{H}_2)$] $\exists\,\lambda_0\geq 1$ such that for all $ x,z\in \R^d$, $\mu\in \pb(\R^d)$,
      \[\lambda_0^{-1}|z|^2\leq \la A z,z\raa\leq \lambda_0 |z|^2,\]
      where $A=(a_{ij})=\sigma\sigma^\ast$, and $\sigma^\ast$ denotes the transpose of the matrix $\sigma$.
\end{itemize}

Let $\wt \Pi$ be the class of functions $F:[0,T]\times \bar\oo\times\pb(\bar\oo)\to U$ such that there exists $C_F>0$
\begin{equation}\label{con-F}
\begin{split}
  &|F(t,x,\mu)-F(t,y,\nu)|\leq C_F\big( |x-y|+\W_2(\mu,\nu)\big),\\
  &\int_0^T|F(s,0,\delta_0)|^2\d s<\infty,\quad t\in [0,T],\, x,y\in \bar\oo,\,\mu,\nu\in \pb(\bar\oo).
\end{split}
\end{equation}

According to  \cite[Theorem 3.2]{Adams} or \cite{Wang21}, under the condition $\mathrm{(H_1)}$, for each $F\in \wt \Pi$ and $\xi\in \F_s$, there exists a unique solution to  the reflected SDE: for $0\leq s\leq t\leq T$,
\begin{equation}\label{b-1}
X_t=\xi+\int_s^tb(r,X_r,\law_{X_r},F(r,X_r,\law_{X_r}))\d r+ \sigma  (B_t-B_s)-\int_s^t\vec{\bf n}(X_r)\d k_r.
\end{equation}

\begin{mydefn}[Admissible feedback controls]\label{defn-2} For $s\in [0,T)$ and $\mu\in \pb(\bar\oo)$,
a control policy ${\bm\alpha}=(\alpha_t)_{t\in [s,T]}$ is said in the class of   admissible feedback controls $\Pi_{s,\mu}$ if there exists a function $F\in \wt \Pi$ such that
\[\alpha_t=F(t,X_t,\law_{X_t}) \]
with  $(X_t)_{t\in [s,T]}$ is the solution to the reflected SDE  \eqref{b-1} with $X_s\in \F_s$ satisfying $\law_{X_s}=\mu$.
\end{mydefn}


We use $(X_t^{s,\mu})_{t\in [s,T]}$ to denote the solution of \eqref{b-1} with initial value $X_s=\xi$ and $\law_{\xi}=\mu$ associated with the admissible feedback control ${\bm \alpha}$.
It follows from the weak uniqueness of \eqref{b-1}, the distribution of $ {X_t^{s,\mu}}$ for $t\in [s,T]$ depends on $\xi$ only through its law $\mu$.
Given two measurable functions $\vartheta:[0,T]\times \bar\oo\times \pb(\bar\oo)\times U\to [0,\infty)$ and $g:\bar\oo\times \pb(\bar\oo)\to [0,\infty)$, our aim is to minimize the objective function
\begin{equation}\label{b-2}
J(s,\mu;\balpha):=\E\Big[\int_s^T \vartheta(r,X_r^{s,\mu},\law_{X_r^{s,\mu}},\alpha_r)\d r+g(X_T^{s,\mu},\law_{X_T^{s,\mu}})\Big].
\end{equation}
We should notice that $J(s,\mu;\balpha)$ is well defined, that is, it depends only on the initial law $\mu$ no matter which random variable $\xi$ or $\tilde \xi$ with $\law_{\xi}=\law_{\tilde \xi}=\mu$ has been used as the initial value of SDE \eqref{b-1}.  Indeed, for $\balpha\in \Pi$ in the form $\alpha_t=F(r,X_t,\law_{X_t})$, we have
\begin{align*}
  \E\Big[\int_s^T \vartheta (r,X_r^{s,\xi},\law_{X_r^{s,\xi}},\alpha_r)\d r\Big]&= \int_s^T\!\!\int_{\bar\oo}\! \vartheta \big(r, x,\law_{X_r^{s,\xi}},F(r,x,\law_{X_r^{s,\xi}})\big)\law_{X_r^{s,\xi}}(\d x)\d r.
\end{align*}
Similar deduction yields that the term $\E\big[g(X_T^{s,\xi},\law_{X_T^{s,\xi}})\big]$ also depends on $\xi$ through its law.

The value function is defined by
\begin{equation}\label{a-2}
V(s,\mu)=\inf_{\balpha\in \Pi_{s,\mu}} J(s,\mu;\balpha).
\end{equation}

Next, we present some properties of the value function. In particular, the value function satisfies the dynamic programming principle, which is based on the flow property of the solution to \eqref{a-1}.

\begin{mylem}\label{lem-1}
Assume $\mathrm{(H_1)}$, $\mathrm{(H_2)}$ hold. For any $U$-valued $\F_t$-adapted process $(\alpha_t)_{t\in [s,T]}$ and $\bar\oo$-valued random variables $\xi,\,\tilde \xi\in \F_s$, consider the solutions $(X_t,k_t)_{t\in [s,T]}$, $(\wt X_t,\tilde k_t)_{t\in [s,T]}$ to \eqref{a-1} with initial values $X_s=\xi$ and $\wt X_s=\tilde \xi$ respectively. Then,
\begin{equation}\label{a-3}
\W_2(\law_{X_t},\law_{\wt X_t})^2\leq \E|X_t-\wt X_t|^2\leq \big(\E |\xi-\tilde \xi|^2\big)\e^{2(K_1+K_1^2)(t-s)},\quad t\in [s,T].
\end{equation}
\end{mylem}

\begin{proof}
  Since $\oo$ is convex and $\vec{\mathbf{n}}(x)$ is unit outward normal of $\oo$, it holds
  \[\la \vec{\mathbf{n}}(x), y-x\raa \leq 0, \quad \forall\, x\in \partial \oo, \ y\in \oo.\]
  By It\^o's formula,
  \begin{align*}
    \d|X_t-\wt X_t|^2&=2\la X_t-\wt X_t,\d ( X_t-\wt X_t)\raa+\d (X_t-\wt X_t)\cdot \d (X_t-\wt X_t)\\& \quad -2\la X_t-\wt X_t, \vec{\mathbf{n}}(X_t)\raa\d k_t + 2\la X_t-\wt X_t, \vec{\mathbf{n}}(\wt X_t)\raa\d \tilde k_t.
  \end{align*}
  Since $k_t$ increases only when $X_t\in \partial \oo$ and $\tilde k_t$ increases only when $\wt X_t\in \partial \oo$, we have
  \[\la X_t-\wt X_t, \vec{\mathbf{n}}(X_t)\raa\d k_t\geq 0,\quad \la X_t-\wt X_t, \vec{\mathbf{n}}(\wt X_t)\raa\d \tilde k_t\leq 0.\]
  Thus, by $\mathrm{(H_1)}$,
  \begin{align*}
    \E|X_t-\wt X_t|^2 \leq \E|\xi-\tilde \xi|^2+ (K_1+K_1^2)\int_s^t \E\big[\big( |X_r-\wt X_r|+\W_2(\law_{X_r},\law_{\wt X_r})\big)^2\big] \d r.
  \end{align*}
  As $\W_2(\law_{X_r},\law_{\wt X_r})^2\leq \E|X_r-\wt X_r|^2$, it follows from Gronwall's inequality that
  \begin{gather*}
    \E|X_t-\wt X_t|^2\leq \E|\xi-\tilde \xi|^2+2(K_1+K_1^2)\int_s^t \E|X_r-\wt X_r|^2\d r,\\
    \W_2(\law_{X_r},\law_{\wt X_r})^2\leq\E|X_t-\wt X_t|^2\leq \big(\E|\xi-\tilde\xi|^2\big)\e^{2(K_1+K_1^2) (t-s)}.
  \end{gather*}
  Therefore, we arrive at the desired estimate \eqref{a-3}.
\end{proof}

Let us introduce the regular condition on the cost functions $\vartheta$ and $g$ as follows.
\begin{itemize}
  \item[$\mathrm{(H_3)}$] There exist $K_2,\,K_3>0$ such that
  \begin{gather*}
  |\vartheta(s,x,\mu,\alpha)|\leq K_2,\quad \forall\, s\in [0,T], x\in \bar\oo, \mu\in\pb(\bar\oo),\alpha\in U;\\
  |\vartheta(s,x,\mu,\alpha)-\vartheta (t,y,\nu,\alpha)|+|g(x,\mu)-g(y,\nu)|\leq
  K_3\big( |s-t|+|x-y|+\W_2(\mu,\nu)\big)
  \end{gather*}
  for all $s,t\in [0,T]$, $x,y\in \bar\oo$, $\mu,\,\nu\in \pb(\bar\oo)$, $\alpha\in U$.
\end{itemize}

\begin{myprop}\label{prop-1}
Suppose $\mathrm{(H_1)}$-$\mathrm{(H_3)}$ hold. Then the value function satisfies
\begin{equation}\label{a-4}
|V(s,\mu)-V(s',\mu')|\leq C\big(\sqrt{|s-s'|}+\W_2(\mu,\mu')\big), \quad s,s'\in [0,T], \, \mu,\mu'\in \pb(\bar\oo),
\end{equation} for some constant $C>0$.
\end{myprop}

\begin{proof}
  Let $0\leq s<s'\leq T$. By the definition of $V(s',\mu')$, for any $\veps>0$ there exists a control $\balpha^\veps\in \Pi_{s', \mu'}$ such that
  \[J(s',\mu')\leq V(s',\mu')+\veps.\]
  Let $(X_t^\veps,k_t^\veps)$ be the associated controlled process to $\balpha^\veps$ with $X_{s'}^\veps=\xi'$ and $\law_{\xi}=\mu'$.
  Due to $\balpha^\veps\in \Pi_{s',\mu'}$, there exists $F^\veps:[0,T]\times\bar\oo\times\pb(\bar\oo)\to U$ in the class $\wt{\Pi}$ such that $\alpha^\veps_r=F^\veps(r,X_r^\veps,\law_{X_r^{\veps}})$ for $r\in [s',T]$.
  Let
  \[\wt F(r,x,\mu)=\begin{cases}
    F^\veps(s',x,\mu), \ &r\in [0,s'],\\
    F^\veps(r,x,\mu),\ &r\in [s',T].
  \end{cases}\]
  We can check directly that $\wt F\in \wt{\Pi}$.
  Consider the reflected SDE
  \begin{gather*}
    \wt X_t=\tilde \xi+\int_s^tb(r,\wt X_r,\law_{\wt X_r},\wt F(r, \wt X_r,\law_{\wt X_r}))\d r+\int_s^t\sigma(r)\d B_r-\int_s^t\vec{\mathbf{n}} (\wt X_r)\d \tilde k_r,\\
    \tilde k_t=\int_{s}^t\mathbf1_{\partial \oo}(\wt X_r)\d \tilde k_r.
  \end{gather*} Here the random variable $\tilde \xi\in \F_{s }$ is chosen so that $\law_{\tilde \xi}=\mu$ and  $\E|\tilde \xi-\xi'|^2=\W_2(\mu,\mu')^2$, whose existence is a result of the existence of optimal coupling of $\mu$ and $\mu'$ (cf. \cite{Vil}). Define
  \[\tilde \alpha_r=\wt F(r, \wt X_r,\law_{\wt X_r}), \quad r\in [s,T],\]
  then $\tilde \balpha=(\tilde \alpha_r)_{r\in [s,T]}$ is in $\Pi_{s,\mu}$.

  Due to  $\mathrm{(H_3)}$ and the compactness of $[0,T]$, $\bar\oo$, $\pb(\bar\oo)$ and $U$, we have $\vartheta$ and $g$ are bounded. Therefore, by Lemma \ref{lem-1},
  \begin{align*}
    &V(s,\mu)-V(s',\mu')\\
    &\leq \E\Big[\int_s^T\!\! \vartheta (r,\wt X_r,\law_{\wt X_r},\tilde \alpha_r)\d r\!+\!g(\wt X_T,\law_{\wt X_T})\Big]\!-\!\E\Big[\int_{s'}^T \!\! \vartheta (r,X_r^\veps,\law_{X_r^\veps},\alpha_r^\veps)\d r\!+\!g(X_T^\veps, \law_{X_T^\veps})\Big] \!+\!\veps\\
    &\leq \E\Big[\int_s^{s'} \vartheta (r,\wt X_r,\law_{\wt X_r},\tilde \alpha_r)\d r\Big]\! +\!2K_3\int_{s'}^T\!\!\big(\E  |\wt X_r \!-\!X_r^\veps|^2 \big)^{\frac 12}\d r\!+\!2K_3\big(\E|\wt X_T-X_T^\veps|^2\big)^{\frac 12}\!+\!\veps\\
    &\leq c_1 |s'-s|+c_1 \big(\E\big[|\wt X_{s'}-\xi|^2\big]\big)^{\frac 12} +\veps\\
    &\leq c_1|s'-s|+ c_1\big(\E[|\wt X_{s'}-\wt X_s|^2]\big)^{\frac 12}+c_1\big(\E[|\wt X_s-\xi|^2]\big)^{\frac 12}+\veps\\
    &\leq c_2\big(|s'-s|+\sqrt{|s'-s|} +\W_2(\mu,\mu')\big)+\veps ,
  \end{align*} where we have used $ \W_2(\law_{\wt X_r},\law_{X_r^\veps})\leq \big(\E|\wt X_r-X_r^\veps|^2\big)^{\frac 12}$, and $c_1$, $c_2$ are constants depending only on $K_3$, $T$ and the diameter of $\oo$.
  Therefore,
  \[V(s,\mu)-V(s',\mu')\leq c_3\big(\sqrt{|s'-s|}+\W_2(\mu,\mu')\big)+\veps,\]
  for some $c_3>0$. Letting $\veps\to 0$, we get the desired estimate of $V(s,\mu)-V(s',\mu')$. The estimate $V(s',\mu')-V(s,\mu)$ can be proved in a similar way. The proof is completed.
\end{proof}

\begin{myprop}[Dynamic programming principle]\label{prop-2}
Suppose that $\mathrm{(H_1)}$, $\mathrm{(H_2)}$ hold. Then, for any $0\leq s\leq t\leq T$, $\mu\in \pb(\bar\oo)$,
\begin{equation}\label{a-5}
V(s,\mu)=\inf_{\balpha\in \Pi_{s,\mu}} \Big\{ \E\Big[\int_s^t \vartheta (r, X_r^{s,\mu,\alpha},\law_{X_r^{s,\mu,\alpha}},\alpha_r)\d r+V(t,\law_{X_t^{s,\mu,\alpha}})\Big]\Big\},
\end{equation}
where for each $\balpha\in \Pi_{s,\mu}$, $(X_t^{s,\mu,\alpha})_{t\in [s,T]}$ stands for the corresponding controlled process with initial value $X^{s,\mu,\alpha}_s$ satisfying $\law_{X^{s,\mu,\alpha}_s}=\mu$.
\end{myprop}

\begin{proof}
  For each $\balpha\in \Pi_{s,\mu}$, the wellposedness of reflected McKean-Vlasov \eqref{a-1} yields that the flow property holds:
  \[X_r^{s,\xi}= X_r^{t, X_t^{s,\xi}},\quad r\in [t,T],\ s\leq t.\]
  This assertion can be proved in the same way as the McKean-Vlasov equations without reflection; see, \cite[Section 3]{Buk}.

  Denote the right-hand side of \eqref{a-5} by $\wt V(s,\mu)$. Then, according to the definition of $V(s,\mu)$, for any $\veps>0$ there exists an admissible feedback control $\balpha\in \Pi_{s,\mu}$ such that
  \begin{align*}
    &V(s,\mu)\\&\geq \E\Big[\int_s^t\!\!\!\vartheta (r,X_r^{s,\mu,\alpha}, \law_{X_r^{s,\mu,\alpha}},\alpha_r)\d r\!+\!\!\int_t^T\!\!\!\vartheta (r,X_r^{s,\mu,\alpha}, \law_{X_r^{s,\mu,\alpha}},\alpha_r)\d r\!+\!g(X_T^{s,\mu,\alpha}, \law_{X_T^{s,\mu,\alpha}})\Big]\!-\!\veps\\
    &\geq \E\Big[\int_s^t \!\vartheta (r,X_r^{s,\mu,\alpha}, \law_{X_r^{s,\mu,\alpha}},\alpha_r)\d r +V(t,\law_{X_t^{s,\mu,\alpha}})\Big]-\veps\\
    &\geq \wt V(s,\mu)-\veps,
  \end{align*}
  where in the second inequality we have used the flow property of $(X_t^{s,\mu,\alpha})$. Letting $\veps\to 0$, we obtain that $V(s,\mu)\geq \wt V(s,\mu)$.

  For the inverse inequality, for any $\veps>0$, by the definition of $\wt V(s,\mu)$, there is a feedback control $\balpha\in \Pi_{s,\mu}$ corresponding to a function $F\in\wt \Pi$ such that
  \begin{equation}\label{b-3}\veps+\wt V(s,\mu)\geq \E\Big[\int_s^t \vartheta (r,X_r^{s,\mu,\alpha}, \law_{X_r^{s,\mu,\alpha}},\alpha_r)\d r+V(t,\law_{X_t^{s,\mu,\alpha}})\Big].
  \end{equation}
  By the definition of $V(t,\law_{X_t^{s,\mu,\alpha}})$, there exists a feedback control $\balpha'\in \Pi_{t,\law_{X_t^{s,\mu,\alpha}}}$ corresponding to a function $\wt F\in \wt\Pi$ such that
  \begin{equation}\label{b-4}
  \veps+V(t,\law_{X_t^{s,\mu,\alpha}})\geq \E\Big[\int_t^T\!\!
   \vartheta (r,X_r^{t,\nu_t},\law_{X_r^{t,\nu_t}},\alpha'_r)\d r+g(X_T^{t,\nu_t},\law_{X_T^{t,\nu_t}})\Big],
   \end{equation}
   where $\nu_t=\law_{X_t^{s,\mu,\alpha}}$.
   We define a new function $\hat F$ by
     \[\hat F(r,x,\mu)=F(r,x,\mu)\mathbf{1}_{r\leq t}+\wt{F}(r,x,\mu)\mathbf1_{t<r\leq T},\]
     and check directly that $\hat{F}\in \wt\Pi$.
   Then, corresponding to $\hat F$, consider the following SDE
   \begin{equation}\label{b-5}
   \d \hat{X}_r=b(r, \hat{X}_r, \law_{\hat{X}_r}, \hat{F}(r,\hat{X}_r,\law_{\hat{X}_r}))\d r+\sigma(r)\d B_r-\vec{\mathbf{n}}(\hat{X}_r)\d k_r
   \end{equation}
   with initial value $\hat{X}_s=\xi$ and $\law_{\xi}=\mu$. By the   uniqueness of solution to SDE \eqref{b-5}, it holds that $\hat{X}_r=X_r^{s,\mu,\alpha}$ for $r\in [s,t]$ and $\hat{X}_r= X_r^{t,\nu_t}$ for $r\in [t, T]$.
   Associated with $\hat{F}$, there is an admissible feedback control $\hat{\balpha} \in \Pi_{s,\mu}$ and $\hat{\balpha}$ satisfies
   \[\hat\alpha_r=\alpha_r\mathbf1_{s\leq r\leq t}+\alpha'_r\mathbf1_{t<r\leq T}.\]
   Then, invoking \eqref{b-3}, \eqref{b-4}, by the definition of $V(s,\mu)$,
   \begin{align*}
     2\veps+\wt V(s,\mu)&\geq \E\Big[\int_s^T\!\!\vartheta (r,\hat X_r , \law_{\hat X_r }, \hat \alpha_r)\d r\ +\!g(\hat X_T ,\law_{\hat X_T })\Big] \geq V(s,\mu).
   \end{align*}
   Letting $\veps\to 0$, we get $\wt V(s,\mu)\geq V(s,\mu)$. In all, we have shown $V(s,\mu)=\wt V(s,\mu)$ and the proof is completed.
\end{proof}

\section{Characterization of the value function: Existence of viscosity solution} \label{sec-3}

\subsection{Riemannain structure of the Wasserstein space}
In this subsection, we adopt the Riemannian interpretation of the Wasserstein space developed by Otto in \cite{Otto} to introduce a HJB equation on the Wasserstein space, and show that the value function is a viscosity solution to it. However, we defer the discussion on the uniqueness of viscosity solution to this HJB equation to Section \ref{sec-4}.

The tangent space, geodesics, and Ricci curvature can be developed on $\pb_2(\R^d):=\{\mu\in \pb(\R^d); \int_{\R^d}|x|^2\d \mu(x)<\infty\}$ endowed with the $L^2$-Wasserstein distance $\W_2$ based on the theory on  optimal transport maps between probability measures. See, for example, the monographs \cite{Amb} and \cite{Vil}. As $\bar\oo$ is bounded, it is clear that $\pb_2(\bar\oo)=\pb(\bar\oo)$.
As we are interested with the reflected stochastic processes on $\bar\oo$,  similar to  $\pb_2(\R^d)$, we consider the following Riemannian structure  of $\pb(\bar\oo)$.
For each $\mu\in \pb(\bar\oo)$, the tangent space at $\mu$ is defined by
\begin{equation}\label{d-1}
\mathscr{T}_\mu:=\big\{ v:\bar\oo\to \R^d\ \text{is measurable satisfying $\mu(|v|^2)<\infty$ and $\la A\nu,\vec{\mathbf{n}}\raa =0$ on $\partial \oo$}\big\},
\end{equation}
where $A=\sigma\sigma^\ast$, $\vec{\mathbf{n}}$ is the unit outward normal of $\oo$.
Then, $\mathscr{T}_\mu$ is a Hilbert space under the inner product
\[\la v,v\raa_{\mathscr{T}_\mu}=\|v\|_{\mathscr{T}_\mu}^2:=\mu(|v|^2) .\]

\begin{mydefn}\label{def-2}
Let $u:\pb(\bar\oo)\to \R$ be a continuous function, and $\mathrm{Id}$ denotes the identity map on $\R^d$. $u$ is said to be intrinsically differentiable at a point $\mu\in \pb(\bar\oo)$, if there is a linear functional $D^{L}u:\mathscr T_\mu\to \R$ such that
\[D_v^L u(\mu)=\lim_{\veps\downarrow 0} \frac{ u(\mu\circ(\mathrm{Id}+\veps v)^{-1})-u(\mu)}{\veps} , \quad v\in \mathscr{T}_\mu,\ \mu\in \pb(\bar\oo).\]
In this situation, the unique element $D^Lu(\mu)\in \mathscr{T}_\mu$ such that
\[ \la D^L u(\mu),v \raa_{\mathscr{T}_\mu}=\int_{\bar\oo}\! \la D^L u(\mu)(x),v(x)\raa \mu(\d x)=D_v^Lu(\mu),\quad v\in \mathscr{T}_\mu. \]
is called the intrinsic derivative of $u$ at $\mu$.

If, moreover,
\[\lim_{\|v\|_{\mathscr{T}_\mu}\to 0} \frac{ u(\mu\circ (\mathrm{Id}+v)^{-1})-u(\mu)-D_v^L u(\mu)}{\|v\|_{\mathscr{T}_\mu}} =0,
\]
then $u$ is called $L$-differentiable at $\mu$ with the $L$-derivative (i.e. Lions' derivative) $D^Lu(\mu)$.
\end{mydefn}

\begin{mydefn}\label{defn-3}
Let $\hat{\pb}$ be a subset of  $\pb(\bar\oo)$.
We write $u\in C_{L,b}^1(\hat\pb)$ if $u$ is Lipschitz continuous in $(\pb(\bar\oo),\W_2)$,  intrinsically differentiable at any point $\mu\in \hat\pb$, and its intrinsic derivative  $D^L u(\mu)(x)$ satisfies that
\begin{itemize}
\item[(i)] for each $\mu\in \hat\pb$, $x\mapsto D^L u(\mu)(x)$ is continuously differentiable;
\item[(ii)] $\sup\big\{|D^L u(\mu)(x)|+|\nabla_x D^L u(\mu)(x)|;\mu\in \hat{\pb}, x\in \bar\oo\big\}<\infty$;
\item[(iii)] $\mu\mapsto D^L u(\mu)(\cdot)$ is continuous from $\hat\pb$ to $L^1(\bar\oo)$ in the sense that if $\mu_n,\mu\in \hat\pb$ and $\W_2(\mu_n,\mu)\to 0$ as $n\to \infty$, then  for any $\veps>0$
    \[\mu_n\big(\{x\in \oo; |\nabla_x D^L\psi(t,\mu_n)(x)-\nabla_x D^L \psi(t,\mu)(x)|\geq \veps\}\big)\longrightarrow 0,\quad \text{as} \ n\to \infty .\]
\end{itemize}
For a function $\psi:[0,T]\!\times\!\hat{\pb}\to \R$, if   for each $\mu\in \hat\pb$, $\psi(\cdot,\mu)$ is continuously differential; for each $t\in [0,T]$,  $\psi(t,\cdot)\in C_{L,b}^1(\hat\pb)$, and
\[\|D^L\psi\|_{\infty}:=\sup\{|D^L\psi(t,\mu)(x)|; t\in [0,T],\mu\in \hat{\pb},x\in\bar\oo\}<\infty,\] we say that $\psi\in C_{L,b}^{1,1}([0,T]\!\times \!\hat\pb)$.
\end{mydefn}

\begin{mydefn}[Absolutely continuous curves]\label{def-4} A curve $\mu:(a,b)\to \pb(\bar\oo)$ is said to be in $AC^q(a,b)$ for $q\in [1,+\infty]$ and $a,b\geq 0$, if there exists $m\in L^q(a,b)$ such that
\[\W_2(\mu_s,\mu_t)\leq \int_s^tm(r)\,\d r,\qquad a<s<t<b.\]
In the case $q=1$ we denote the space simply by $AC(a,b)$ and we are dealing with absolutely continuous curves.

For an absolutely continuous curve $\mu:(a,b)\to \pb(\bar\oo)$ the limit
 \[|\mu'|(t):=\lim_{s\to t}\frac{\W_2(\mu_s,\mu_t)} {|s-t|}\]
 exists for $Leb$-a.e. $t\in (a,b)$, which is called the metric derivative of the curve $(\mu_t)$.
\end{mydefn}

Next, let us recall some results on the absolutely continuous curves in $\pb(\bar\oo)$ as a subspace of $\pb_2(\R^d)$, which can be proved in the same way as in $\pb_2(\R^d)$ with some necessary modifications.

\begin{mythm}[\cite{Amb}, Theorem 8.3.1]\label{Amb}
Let $\mu:[0,T]\to \pb(\bar\oo)$ be an absolutely continuous curve and let $|\mu'|\in L^1([0,T])$ be its metric derivative. Then there exists a Borel vector field $v:(t,x)\mapsto v_t(x)$ such that
\begin{equation}\label{d-2}
v_t\in L^2(\bar\oo\to\R^d;\mu_t),\quad \|v_t\|_{L^2(\bar\oo;\mu_t)}\leq |\mu'|(t)\quad \text{for a.e. $t\in [0,T]$,}
\end{equation}
and the continuity equation
\begin{equation}\label{d-3}
\partial_t\mu_t +\nabla \cdot(v_t\mu_t)=0\quad \ \text{in $[0,T]\times \bar\oo$}
\end{equation}
holds in the sense of distribution, i.e.
\begin{equation}\label{d-4}
\int_0^T\!\int_{\oo}\!\!\Big(\partial_t\psi(t,x)\!+\! \la v_t(x),\nabla_x\psi(t,x)\raa\Big)\d \mu_t(x)\d t =0\quad \forall\,\psi \in C_c^\infty((0,T)\times \oo),
\end{equation} where $C_c^\infty((0,T)\times \oo)$ stands for the set of smooth functions on $(0,T)\times\oo$ with compact support.

Conversely, if a continuous curve $\mu:[0,T]\to \pb(\bar\oo)$ satisfies the continuity equation \eqref{d-3} for some Borel velocity field $v_t$ with $\|v_t\|_{L^2(\bar\oo;\mu_t)}\in L^1([0,T])$, then
$\mu:[0,T]\to \pb(\bar\oo)$ is absolutely continuous and $|\mu'|(t)\leq \|v_t\|_{L^2(\bar\oo;\mu_t)}$ for $Leb$-a.e. $t\in [0,T]$.
\end{mythm}

\begin{myprop}[\cite{Amb}, Theorem 8.4.6]\label{prop-1.5}
Let $\mu:[0,T]\to \pb(\bar\oo)$ be absolutely continuous, and let $v_t\in \mathscr{T}_{\mu_t}$ be such that \eqref{d-2}, \eqref{d-3} hold. Then
\begin{equation}\label{d-5}
\lim_{\veps\downarrow 0}\frac{\W_2(\mu_{t+\veps},\mu_t\circ (\mathrm{Id} +\veps v_t)^{-1})}{\veps}=0.
\end{equation}
\end{myprop}

It is easy to show that the curve $(\law_{X_t})_{t\in [s,T]}$ of the controlled process $(X_t)_{t\in [s,T]}$ is an absolutely continuous curve in $\pb(\bar\oo)$ under the coefficients conditions $\mathrm{(H_1)}$, $\mathrm{(H_2)}$.  However, it is a hard work to describing  the tangent vector field of $(\law_{X_t})_{t\in [s,T]}$ in $\pb(\bar\oo)$.  It is the basis to characterize the value function via the theory of HJB equation. In this work, in order to describe the tangent vector field along the distribution curve of the controlled process, we need to restrict the initial distribution to be in the  regular subspace $ \pb^r(\bar\oo)$ defined in \eqref{b-0} of $\pb(\bar\oo)$.

Let $\mu_0\in \pb^r(\bar\oo)$, and $\xi$ is a random variable in $\F_0$ with $\law_\xi=\mu_0$. For $\balpha\in \Pi_{0,\mu_0}$, denote $(X_t^{0,\mu_0},k_t^{0,\mu_0})_{t\in [0,T]}$ its associated controlled process satisfying \eqref{a-1}. Under the nondegenerate condition $\mathrm{(H_2)}$, the law of $X_t^{0,\mu_0}$ admits a density $\rho_t(x)$, which satisfies the nonlinear Fokker-Planck  equation:
\begin{equation}\label{d-6}
\begin{cases}
  \partial_t\rho_t(x)= \mathcal{L}^\ast_\alpha\rho_t(x)
  ,\quad & x\in \oo, t\in (0,T),\\
  \la A\nabla  \rho_t(x), \vec{\mathbf{n}} (x)\raa=0,\qquad \qquad &x\in \partial \oo, t\in (0,T).
\end{cases}
\end{equation}
where
\[ \mathcal{L}^\ast_\alpha\rho_t(x)=\frac 12\sum_{i,j=1}^d a_{ij}\partial^2_{ x_i x_j} \rho_t(x) -\sum_{i=1}^d \partial_{ x_i}\big(b_i(t,x,\rho_t(x)\d x,\alpha_t)\rho_t(x)\big).\]

Using the decoupling method, via fixing the distribution of the process $(X_t)$, the controlled process $(X_t,k_t)$ can be viewed as a solution to the following SDE
\begin{equation*}
\begin{cases}
  \d \wt X_t=b(t,\wt X_t,\mu_t,\alpha_t)\d t+\sigma \d B_t-\vec{\mathbf{n}}(\wt X_t)\d \tilde k_t,\\
  \tilde k_t=\int_0^t\mathbf1_{\partial \oo}(\wt X_s)\d \tilde k_s,
\end{cases},
\end{equation*}
where $\mu_t=\law_{X_t}$ is fixed by the unique solution of  SDE \eqref{a-1}. In particular, the law of $X_t$ coincides with that of $\wt X_t$. Let $p(s,x; t,y)$ be the transition probability of the process $(\wt X_t)$, which is a fundamental solution of a parabolic equation with Neumann boundary condition. There is a large number of literatures on the estimates of fundamental solutions to parabolic equations with Dirichlet boundary condition or Neumann boundary condition; see, for instance, \cite{Aro,BH,Chou,Fri,YZ,Zh96,Zh97} and references therein. In particular,  \cite{YZ} generalized the work \cite{Zh97} to the time-homogeneous parabolic equation with mixed boundary condition. \cite{Chou} deals with time-inhomogeneous parabolic equations with Neumann  boundary. Under $(\mathrm{H_1})$, the drift $b$ admits a bound $M$ determined by $K_1$ and the diameters of $\oo$ and $U$.   Then, the Gaussian type estimates hold for $p(s,x;t,y)$. Namely, there exist constants $\kappa_1,\,\kappa_2>0$, depending on $T$, such that
\begin{equation}\label{heat}
 \begin{aligned}
  &\frac{1}{\kappa_1 (t-s)^{d/2}} \exp\Big(- \frac{|y-x|^2}{\kappa_2(t-s)}\Big)\leq p(s,x;t, y)\leq \frac{\kappa_1}{(t-s)^{d/2}}\exp\Big(-\kappa_2 \frac{|y-x|^2}{t-s}\Big),\\
   &|\partial_t p(s,x;t, y)|\leq \frac{\kappa_1}{(t-s)^{(d+2)/2}}\exp\Big(-\kappa_2 \frac{|y-x|^2}{t-s}\Big), \qquad x,y\in \bar\oo, 0\leq s<t\leq T.
 \end{aligned}
\end{equation}
Furthermore, the density $\rho_t(x)$ of $\law_{X_t}$ can be represented by
\begin{equation}\label{rep}
\rho_t(x)=\int_{\oo} p(s,z;t,x)\mu_0(\d z),\quad t>s.
\end{equation}
Consequently,   under the nondegenerate condition $\mathrm{(H_2)}$,  the distribution of the solution $X_t$  to  SDE \eqref{a-1} will always stay in $\pb^r(\bar\oo)$.
This makes it feasible to characterize the value function as a viscosity solution to certain  HJB equation on $\pb^r(\bar\oo)$.

\begin{mythm}[Characterization of tangent vector fields: regular case] \label{thm-3.6}
Assume $\mathrm{(H_1)}$ and $\mathrm{(H_2)}$ hold.  Let $(X_t, k_t)_{t\in [0,T]}$ be a solution to \eqref{a-1} associated with a feedback control $\balpha \in \Pi_{0,\mu_0}$ in the form $\alpha_t=F(t,X_t,\law_{X_t})$ and $\law_{X_0}=\mu_0 \!\in \! \pb^r(\bar\oo)$.  Then,
\begin{itemize}
\item[$\mathrm{(i)}$] $[0,T]\ni t\mapsto \mu_t:=\law_{X_t}$ is an absolutely continuous curve in $\pb^r(\bar\oo)$. Its associated velocity field $v_t$ satisfying \eqref{d-2}, \eqref{d-3} is given by
    \begin{equation}\label{d-7}
    v_t(x )=\sum_{i=1}^d\Big( b_i(t,x,\mu_t,F(t,x,\mu_t))\!-\!\frac12 \sum_{j=1}^d \frac{a_{ij} \partial_{x_j}  \rho_t(x) } {\rho_t(x)}\Big)\mathbf{e}_i
    \end{equation}
    where  $\rho_t(x)=\frac{\d \mu_t(x)}{\d x}$ denotes the density of $ \mu_t $; $\{\mathbf{e}_1,\ldots,\mathbf{e}_d\}$ is the canonical orthonormal basis of $\R^d$.
\item[$\mathrm{(ii)}$] Let  $u\in C_{L,b}^{1}(\pb^r(\bar\oo))$, then
\begin{equation}\label{d-8}
\begin{split}
 \frac{\d u(\mu_t)}{\d t}& =D_{v_t}^L u(\mu_t)\\
 &=\int_{\bar\oo}\la b(t,x,\mu_t,\alpha_t), D^L u(\mu_t)(x)\raa \d\mu_t(x)\!-\!\frac 12\int_{\bar\oo}\! \la D^L u(\mu_t)(x), A\nabla_x\rho_t(x)\raa \d x\\
 &=\int_{\bar\oo}\!\la b(t,x,\mu_t,\alpha_t),D^L u(\mu_t)(x)\raa \d \mu_t(x)\!+\!\frac12 \int_{\bar \oo}\!\mathrm{tr}\big(A\nabla_x D^L u(\mu_t)(x)\big)\d \mu_t(x).
  \end{split}
\end{equation}
\end{itemize}
\end{mythm}

\begin{proof}
  (i) By $\mathrm{(H_1)}$ and Lemma \ref{lem-1},  we have that for $0\leq s<t\leq T$
  \begin{align*}
\W_2&(\mu_t,\mu_s)^2\leq \E|X_t-X_s|^2\\
    &\leq C\Big(\E \Big|\int_s^t\!b(r,X_r,\law_{X_r},\alpha_r)\d r\Big|^2+\E\big| \sigma   ( W_t-W_s)\big|^2\Big)\\
    &\leq C\Big((t-s) \E\int_s^t\!(1+|X_r|^2)\d r+  \|\sigma\|^2(t-s) \Big)\\
    &\leq C\big(|t-s|^2+|t-s| \big)
  \end{align*} for a generic positive constant $C$ whose value may change from line to line. Therefore, $t\mapsto\mu_t$ is absolutely continuous in $\pb(\bar\oo)$. Theorem \ref{Amb} implies that there exists a velocity $v_t$ such that \eqref{d-2} and \eqref{d-3} hold.


  Due to \eqref{d-4}, for any $\psi(t,x)=\beta(t)h(x)\in C_c^\infty((0,T)\times \bar\oo)$,
  \begin{align*}
    &\int_0^T\!\int_{\oo}\!\big(\beta'(t)h(x)\!+\!\la v_t(x),\beta(t)\nabla_x h(x)\raa\big)\d \mu_t(x)\d t\\
    &=\int_0^T\!\Big(\beta'(t)\E h(X_t)\!+\! \beta(t)\E\big[\la v_t(X_t),\nabla_x h(X_t)\raa\big]\Big)\d t=0.
  \end{align*}
  This yields
  \begin{equation}\label{d-9} \int_0^T\beta(t)\frac{\d}{\d t}\E h(X_t)\,\d t= -\int_0^T\beta'(t)\E h(X_t)\, \d t=\int_0^T\!\beta(t)\E\big[\la v_t(X_t),\nabla_x h(X_t)\raa\big] \d t.
  \end{equation}
  Applying It\^o's formula and then Green's formula, for $h\in C_c^\infty(\oo)$, we have
  \begin{align*}
\frac{\d }{\d t}\E h(X_t)&=\E\big[\la b(t,X_t,\mu_t,\alpha_t),\nabla_x h(X_t)\raa \!+\! \frac 12 \mathrm{tr}\big(A \nabla_x^2 h(X_t) \big )\big]\\
    &=\E\big[\la b(t,X_t,\mu_t,\alpha_t),\nabla_x h(X_t)\raa\big] -\frac12\E\Big[\sum_{i=1}^d\frac{\sum_{j=1}^d\!a_{ij} \partial_{x_j} \rho_t(X_t) }{\rho_t(X_t)} \partial_{x_i}h(X_t)\Big].
  \end{align*}
  Inserting this into the left-hand side of \eqref{d-9}, the arbitrariness of $\beta(t) h(x)\in\! C_c^\infty((0,T)\!\times \!\oo)$ can yield  that $v_t(x)$ can be represented as   \eqref{d-7}.

  (ii) Since $u$ is Lipschitz continuous in $(\pb(\bar\oo),\W_2)$, there exists $C>0$ such that
  \[|u(\mu_{t+\veps})-u(\mu_t\circ (\mathrm{Id}+\veps v_t)^{-1})|\leq C\W_2(\mu_{t+\veps}, \mu_t\circ (\mathrm{Id}+\veps v_t)^{-1})=o(\veps).\]
  According to Proposition \ref{prop-1.5},
  \[\W_2(\mu_{t+\veps},\mu_t\circ (\mathrm{Id}+\veps v_t)^{-1})=o(\veps),\]
  where $v_t$ is given by \eqref{d-7}. Thus,
  \begin{align*}
    \lim_{\veps\downarrow 0} \frac{u(\mu_{t+\veps})-u(\mu_t)}{\veps} &=\lim_{\veps\downarrow 0} \frac{ u(\mu_t\circ (\mathrm{Id}+\veps v_t)^{-1})-u(\mu_t)}{\veps}=\la D^L u(\mu_t),v_t\raa_{\mathscr{T}_{\mu_t}}\\
    &=\int_{\bar\oo}\la b(t,x,\mu_t,\alpha_t), D^L u(\mu_t)(x)\raa \d\mu_t(x)\!-\!\frac 12\int_{\bar\oo}\! \la D^L u(\mu_t)(x), A\nabla_x\rho_t(x)\raa \d x.
  \end{align*}
  Since $\la A D^L u(\mu), \mathbf{n}\raa =0$ on $\partial \oo$, we derive from Green's formula that
  \[\la D^L u(\mu_t),v_t\raa_{\mathscr{T}_{\mu_t}} =\int_{\bar\oo}\la b(t,x,\mu_t,\alpha_t), D^L u(\mu_t)(x)\raa \d\mu_t(x)\!+\!\frac 12\int_{\bar\oo}\! \mathrm{tr}\big( A\nabla_x D^L u(\mu_t)(x)\big) \d\mu_t(x).\]
  We complete the proof.
\end{proof}

\begin{mythm}[Characterization of tangent vector fields: general case] \label{thm-3.7} Assume $\mathrm{(H_1)}$ and $\mathrm{(H_2)}$ hold.
Let $(X_t, k_t)_{t\in [0,T]}$ be a solution to \eqref{a-1} associated with a feedback control $\balpha \in \Pi_{0,\mu_0}$ in the form $\alpha_t=F(t,X_t,\law_{X_t})$ and $\law_{X_0}=\mu_0\!\in\! \pb(\bar\oo)$.  Then, for  every $u\! \in \! C^1_{L,b}(\pb(\bar\oo))$  and $0\leq t_1<t_2\leq T$, it holds
\begin{equation}\label{d-6.5}
u(\mu_{t_2})-\mu(\mu_{t_1})=\int_{t_1}^{t_2}\!\!\!\int_{\bar\oo}\! \Big( \la b(t,x,\mu_t,\alpha_t), D^L u(\mu_t)(x)\raa\! +\frac 12\mathrm{tr}\big(A \nabla_xD^L u(\mu_t)(x)\big)\Big)\d \mu_t(x) \d t.
\end{equation}
\end{mythm}

\begin{proof}
Similar to proof of Theorem \ref{thm-3.6}(i), for $\mu_{0}\in \pb(\bar\oo)$ instead of in $\pb^r(\bar\oo)$, the curve $(\mu_t)$ is still an absolutely continuous curve in $\pb(\bar\oo)$. The existence of the vector field $v_t $ satisfying \eqref{d-3} in the sense of \eqref{d-4} still exists according to Theorem \ref{Amb}. Now, we cannot have the explicit expression \eqref{d-7} for $v_t$. Nevertheless, by \eqref{d-4}, similar to the deduction in \eqref{d-9}, using It\^o's formula and smooth approximation, it  holds that for any $\psi(t,x)\in C^{0,2}([0,T]\times\bar\oo)$ (i.e. $\psi(t,x)$ is continuous in $t$ and   second order continuously differentiable in $x$), for $0\leq t_1<t_2\leq T$,
\begin{equation}\label{d-10}
\begin{split}
&\int_{t_1}^{t_2}\int_{\bar\oo}\la v_t(x), \nabla \psi(t,x)\raa \d\mu_t(x)\d t=\int_{\bar\oo}\psi(t_1,x)\d\mu_{t_1}(x)-\int_{\bar\oo}\psi(t_2,x) \d\mu_{t_2}(x) \\
&\qquad \qquad + \int_{t_1}^{t_2}\int_{\bar\oo}\Big( \la b(t,x,\mu_t,\alpha_t),\nabla_x \psi(t,x)\raa+\frac12 \mathrm{tr}\big(A\nabla^2_x\psi(t,x)\big) \Big)\d \mu_t(x)\d t.
\end{split}
\end{equation}
The relation
 \[\W_2(\mu_{t+\veps},\mu_t\circ (\mathrm{Id}+\veps v_t)^{-1})=o(\veps)\] still holds due to Proposition \ref{prop-1.5}, and hence for $u\in C^1_{L,b}(\pb(\bar\oo))$,
 \begin{equation*}
   \begin{split}
    u(\mu_{t_2})-u(\mu_{t_1})= \int_{t_1}^{t_2}  \frac{\d u(\mu_t)}{\d t} \d t=\int_{t_1}^{t_2} \!\! \int_{\bar\oo}\! \la D^L u(\mu_t)(x), v_t(x) \raa\d \mu_t(x) \d t.
   \end{split}
 \end{equation*} Define
 \[\psi(t,x)=\int_{x_0}^x D^Lu (\mu_t)(x)\d x+M,\quad \text{for some $x_0\in \oo$},\]
 where $M$ is a constant such that $\int_{\bar\oo}\psi(0,x)\d \mu_0(x)=\int_{\bar\oo}\psi(T,x)\d \mu_T(x)$. Since $u\in C^1_{L,b}(\pb(\bar\oo))$,
 we have $\psi\in C^{0,2}([0,T]\times\bar\oo)$
 and
 \begin{equation}\label{d-11}
 \nabla_x \psi(t,x)=D^Lu(\mu_t)(x),\quad x\in \bar\oo, t\in (0,T),
 \end{equation}
 then we derive from \eqref{d-10} that
 \begin{equation*}
 u(\mu_{t_2})-\mu(\mu_{t_1})=\int_{t_1}^{t_2}\!\!\!\int_{\bar\oo}\! \Big( \la b(t,x,\mu_t,\alpha_t), D^L u(\mu_t)(x)\raa\! +\frac 12\mathrm{tr}\big(A \nabla_xD^L u(\mu_t)(x)\big)\Big)\d \mu_t(x) \d t,
 \end{equation*} which is the desired conclusion.
 \end{proof}

\subsection{Viscosity solutions to HJB equations}

Based on Theorem \ref{thm-3.6}, applying the dynamic programming principle, Proposition \ref{prop-2}, we shall characterize the value function as a unique viscosity solution to the following HJB equation:
\begin{equation}\label{e-1}
\begin{cases}
- \partial_t u(t,\mu)\!-\!\inf_{\alpha\in U} \mathcal{H}(t,\mu,u,D^L u,\alpha )\!=0,   &t\in [0,T), \mu\in \pb (\bar\oo),\\
u(T,\mu)=\int_{\oo}g(x,\mu)\d \mu(x), &\quad\mu\in \pb (\bar\oo),
\end{cases}
\end{equation} where the Hamiltonian
\begin{equation}\label{e-v}
\begin{split}
 \mathcal{H}(t,\mu,u,D^L u,\alpha )
 & = \int_{\oo} \la b(t,x,\mu,\alpha), D^Lu(t,\mu)\raa \d \mu(x)\\
 &\quad+\!\frac 12\!\int_{\oo}\!\mathrm{tr} \big( A \nabla_xD^L u(t,\mu)(x)\big) \d \mu(x) \!+\!\int_{\oo}\! \vartheta (t,x,\mu,\alpha)\d \mu(x).
 \end{split}
\end{equation}

Let us first introduce the notation of viscosity solution for the equation \eqref{e-1}.

\begin{mydefn}\label{def-5}
Let $u:[0,T]\times \pb (\bar\oo)\to \R$ be a continuous function.
\begin{itemize}
  \item[$(i)$] $u$ is called a viscosity subsolution to \eqref{e-1} if $u(T,\mu)=\int_{\oo}g(x,\mu)\d\mu(x)$, and
      \begin{equation}\label{e-2}
      -\partial_t\psi(t_0,\mu_0)-\inf_{\alpha\in U}\mathcal{H}\big(t_0,\mu_0, \psi ,  D^L \psi,\alpha \big)\leq 0
      \end{equation} for all $\psi\in C_{L,b}^{1,1}\big([0,T)\!\times\!\pb (\bar\oo)\big)$ and all $(t_0,\mu_0)\in [0,T)\!\times\! \pb (\bar\oo)$ being a maximum point of $u-\psi$.

  \item[$(ii)$] $u$ is called a viscosity supersolution to \eqref{e-1} if $u(T,\mu)=\int_{\oo}g(x,\mu)\d\mu(x)$, and
      \begin{equation}\label{e-3}
      -\partial_t\psi(t_0,\mu_0)-\inf_{\alpha\in U}\mathcal{H}\big(t_0,\mu_0,\psi ,D^L \psi,\alpha \big)\geq 0
      \end{equation} for all $\psi\in C_{L,b}^{1,1}\big([0,T)\!\times\!\pb (\bar\oo)\big)$ and all $(t_0,\mu_0)\in [0,T)\!\times \!\pb (\bar\oo)$ being a minimum point of $u-\psi$.
  \item[(iii)] If $u$ is both a viscosity subsolution and a viscosity supersolution to equation \eqref{e-1}, then $u$ is called a viscosity solution to \eqref{e-1}.
\end{itemize}
\end{mydefn}

\begin{mylem}\label{lem-2}
Assume $\mathrm{(H_1)}$-$\mathrm{(H_3)}$ hold. Let  $\psi\in C_{L,b}^{1,1}\big([0,T)\!\times\!\pb (\bar\oo)\big)$. If  $\mu$,  $\mu_n\in \pb (\bar\oo)$, $n\geq 1$, satisfy $\lim_{n\to \infty} \W_1(\mu_n,\mu)=0$, then
\begin{equation}\label{e-3.1}
\lim_{n\to\infty} \mathcal{H}(t,\mu_n,\psi,D^L\psi,\alpha)= \mathcal{H}(t,\mu,\psi,D^L\psi,\alpha)
\end{equation}
uniformly w.r.t.\,$\alpha\in\! U$.
\end{mylem}


\begin{proof}
  We shall estimate the convergence of three terms in $\mathcal{H}(t,\mu_n,\psi,D^L\psi,\alpha)$ separately.
  We shall note that as $\bar\oo$ is compact,  for every $p\geq 1$, $ \lim_{n\to \infty} \W_p(\mu_n,\mu)=0$ is  equivalent to the weak convergence of $\mu_n$ to $\mu$ (cf. \cite[Chapter 6]{Vil}).

  Firstly, consider the convergence of the term
  \begin{equation}\label{e-3.5}
  \begin{aligned}
    &\int_{\oo}\!\la D^{L}\psi(t,\mu)(x),b(t,x,\mu,\alpha)\raa \d \mu(x)\!-\!\int_{\oo}\!\la D^L\psi(t,\mu_n)(x),b(t,x,\mu_n, \alpha)\raa\d\mu_n(x)\\
    &=\int_{\oo}\la D^L\psi(t,\mu)(x)-D^{L}\psi(t,\mu_n)(x),b(t,x,\mu,\alpha)\raa \d \mu_n(x)\\
    &\quad+\int_{\oo} \la D^L\psi(t,\mu_n)(x), b(t,x,\mu_n,\alpha)- b(t,x,\mu,\alpha)\raa \d\mu_n(x)\\
    &\quad +\int_{\oo} \la D^L\psi(t,\mu)(x),b(t,x,\mu,\alpha) \raa \big(\d \mu(x)-\d \mu_n(x)\big)\\
    &=:\mathrm{(I_1)}+\mathrm{(I_2)}+\mathrm{(I_3)}.
  \end{aligned}
  \end{equation}
  Put
  \begin{gather*}
  M_b=\sup\big\{ |b(t,x,\nu,\alpha)|;(t,x,\nu,\alpha)\in [0,T]\times \bar\oo\times\pb(\bar\oo)\times U\big\},\\
  M_{\psi}=\sup\big\{ |D^L\psi(t,\nu)(x)|+|\nabla_x D^L\psi(t,\nu)(x)|; (t,x,\nu)\in [0,T]\times \bar\oo\times\pb (\bar\oo)\big\},
   \end{gather*}
   which are all finite due to $\mathrm{(H_1)}$,  $\psi\in C_{L,b}^{1,1}\big([0,T)\!\times\!\pb (\bar\oo)\big)$, and the compactness of $[0,T],\,\bar\oo,\pb(\bar\oo)$ and $U$.
  Then,
  \[\big|\mathrm{(I_1)}\big|\leq \int_{\oo} M_b|D^L\psi(t,\mu)(x)-D^L\psi(t,\mu_n)(x)|\d\mu_n(x).\]
  By  Definition \ref{defn-3}(iii),
   \begin{equation}\label{e-3.6}
   \mu_n\big(\{x\in \oo; |D^L\psi(t,\mu_n)(x)-D^L\psi(t,\mu)(x)|\geq \veps\}\big)\longrightarrow 0,\quad \text{as}\ n\to\infty,
   \end{equation}
   and hence, the dominated convergence theorem yields that
  \begin{equation}\label{e-4}
  \begin{split}
  \lim_{n\to\infty} \big|\mathrm{(I_1)}\big|&\leq \lim_{n\to\infty}\int_{\oo}M_b   |D^L\psi(t,\mu)(x)-D^L\psi(t,\mu_n)(x)|\d \mu_n(x)
   =0,
  \end{split}
  \end{equation} uniformly w.r.t. $\alpha$. Next, for term $\mathrm{(I_2)}$,
  it follows from  $\mathrm{(H_1)}$ that
  \begin{equation}\label{e-5}
  \lim_{n\to\infty}\big|\mathrm{(I_2)}\big|\leq \!\lim_{n\to \infty} K_1 M_\psi  \W_2(\mu_n,\mu) =0,\quad
  \text{uniformly w.r.t.\,$\alpha\in U$}.
  \end{equation}

  Now we proceed to estimate the term $\mathrm{(I_3)}$.
  Under the condition $\mathrm{(H_1)}$, one can check directly that
  $x\mapsto\la D^L\psi(t,\mu)(x), b(t,x,\mu,\alpha)\raa$ is a bounded, Lipschitz continuous function with
  \[ \sup_{\alpha\in U}\sup_{x\neq y}\frac{ \la D^L\psi(t,\mu)(x), b(t,x,\mu,\alpha)\raa- \la D^L\psi(t,\mu)(y), b(t,y,\mu,\alpha)\raa}{|x-y|}<\infty.\] According to the dual representation of Wasserstein distance $\W_1$, i.e.
  \begin{equation}\label{dual}
  \W_1(\mu,\nu)=\sup\Big\{ \int_{ \oo} h(x)\d \mu(x)-\int_{ \oo} h(x)\d \nu(x);\ \sup_{x\neq y} \frac{|h(x)-h(y)|}{|x-y|}\leq 1\Big\},
  \end{equation}
  there is some constant $C>0$ such that
  \begin{equation}\label{e-6}
  |\mathrm{(I_3)}|\leq C \W_1(\mu_n,\mu)\longrightarrow 0,\quad \text{as}\ n\to \infty,\quad \text{uniformly w.r.t.\,$\alpha\in U$}.
  \end{equation}
  Inserting the estimates \eqref{e-4}, \eqref{e-5}, \eqref{e-6} into \eqref{e-3.5}, we get
  \begin{equation}\label{e-7}
  \lim_{n\to \infty}
  \int_{\oo}\!\la D^{L}\psi(t,\mu)(x),b(t,x,\mu,\alpha)\raa \d \mu(x)\!-\!\int_{\oo}\!\!\la D^L\psi(t,\mu_n)(x),b(t,x,\mu_n, \alpha)\raa\d\mu_n(x)=0
  \end{equation} uniformly w.r.t.\,$\alpha\in U$.

  Secondly,
  \begin{equation*}
  \begin{aligned}
    &\Big|\int_{\oo}\!\mathrm{tr}\big( A\nabla_x D^L\psi(t,\mu_n)(x)\big) \d\mu_n(x) -\int_{\oo}\mathrm{tr}\big( A\nabla_x D^L\psi(t,\mu)(x)\big)\d \mu(x)\Big|\\
    &\leq \Big|\int_{\oo}\big[\mathrm{tr}\big( A\nabla_x  D^L\psi(t,\mu_n)(x)\big)-\mathrm{tr}\big( A\nabla_x D^L \psi(t,\mu)(x)\big)\big] \d\mu_n(x)\Big|\\
    &\quad +\Big|\int_{\oo} \mathrm{tr}\big(A \nabla_x D^L\psi(t,\mu)(x)\big) \d(\mu_n-\mu)(x)\Big|.
  \end{aligned}
  \end{equation*}
   Then, by  Definition \ref{defn-3}(iii) and the weak convergence of $\mu_n$ to $\mu$, we get that
   \begin{equation}\label{e-8}
   \lim_{n\to \infty} \Big|\int_{\oo}\!\mathrm{tr}\big( A\nabla_x D^L\psi(t,\mu_n)(x)\big) \d\mu_n(x) -\int_{\oo}\mathrm{tr}\big( A\nabla_x D^L\psi(t,\mu)(x)\big)\d \mu(x)\Big|=0.
   \end{equation}


  At last, due to $\mathrm{(H_3)}$ and the dual representation \eqref{dual} of $\W_1$,
  \begin{equation}\label{e-9}
  \begin{aligned}
    &\Big|\int_{\oo} \vartheta(t,x,\mu_n,\alpha)\d \mu_n(x) \!- \!\int_{\oo} \vartheta(t,x,\mu,\alpha)\d \mu(x) \Big|\\
    & \leq \int_{\oo} |\vartheta(t,x,\mu_n,\alpha)-\vartheta(t,x,\mu,\alpha)|\d \mu_n(x)+\Big|\int_{\oo} \vartheta(t,x,\mu,\alpha) \d (\mu_n -  \mu)(x)\Big|\\
    &\leq K_3\W_2(\mu_n,\mu)+K_3\W_1(\mu_n,\mu)\longrightarrow 0,\quad \text{as $n\to \infty$, uniformly w.r.t.\,$\alpha\in U$.}
  \end{aligned}
  \end{equation}
  Consequently, the desired conclusion \eqref{e-3.1} follows immediately from \eqref{e-9}, \eqref{e-8} and \eqref{e-7}. The proof is complete.
\end{proof}

%
%


\begin{mylem}\label{lem-5} Assume $\mathrm{(H_1)}$ and $\mathrm{(H_2)}$ hold. Then for any $\mu_{t_0}\in \pb(\bar\oo)$ and $\Theta\in \Pi_{t_0,\mu_{t_0}}$, the law $\law_{X_t}$ of the controlled process $X_t$ satisfies that for any $t_0\leq s<t\leq T$,
\[\lim_{t\to s}\W_1(\law_{X_s},\law_{X_t})=0.
\]
\end{mylem}

\begin{proof}
  If $s=t_0$,
  for any Lipschitz continuous function $h$ with
  $|h|_{Lip}:= \sup\limits_{x\neq y}\frac{|h(x)-h(y)|}{|x-y|}\leq 1$,
  by \eqref{rep} and \eqref{heat},
  \begin{align*}
    &\Big|\int_{\bar\oo} h(x) \d \mu_t(x)-\int_{\bar\oo} h(x)\d \mu_{t_0}(x)\Big|\\
    &=\Big|\int_{\bar\oo}\int_{\bar \oo} h(x)p(t_0,y; t,x)\d \mu_{t_0}(y)\d x-\int_{\bar\oo}h(y)\d \mu_{t_0}(y)\Big|\\
    &\leq \int_{\bar \oo}\int_{\bar\oo} p(t_0,y; t,x)|h(x)-h(y)|\d x\d \mu_{t_0}(y)\Big|\\
    &\leq \int_{\bar\oo}\int_{\R^d}\frac{\kappa_1} {|t-t_0|^{\frac d2}} \e^{-\kappa_2\frac{|x-y|^2}{t-t_0}}|h(x)-h(y)|\d x\d \mu_{t_0}(y) \\
    &=\int_{\bar\oo}\int_{\R^d} \kappa_1\e^{-\kappa_2 |z|^2} |h(y+\sqrt{t-t_0} z)-h(y)|\d z\d \mu_{t_0}(y)\\
    &\leq \kappa_1 \sqrt{t-t_0} \int_{\R^d}  |z| \e^{-\kappa_2|z|^2}\d z,
  \end{align*}
  which tends to $0$ as $t-t_0\to 0$.
  Thus, the dual representation \eqref{dual} of $\W_1$ yields
  \[\lim_{t\to t_0}\W_1(\mu_t,\mu_{t_0})=0.\]

  If $s>t_0$, for any  continuous function $h$ with $|h|_{\infty}:=\sup_{x\in \R^d} |h(x)|\leq 1$,
  \begin{align*}
    &\Big|\int_{\bar\oo} h(x)\d \mu_t(x)-\int_{\bar\oo}h(x)\d \mu_s(x) \Big|\\
    &\leq \Big|\int_{\bar\oo} \!\int_{\bar\oo}\! |h(x)| \big|p(t_0, y; t,x)-p(t_0,y;s,x)\big|\d x\d \mu_{t_0}(y)\\
    &\leq \int_{\bar\oo}\int_{\bar\oo}\!\int_s^t \big|\partial_r p(t_0,y;r,x)\big|\d r\d x\d\mu_{t_0}(y)\\
    &\leq \int_{\bar\oo}\int_{\R^d} \int_{s}^t \frac{\kappa_1}{(r-t_0)^{\frac{d}{2}+1}}\e^{-\kappa_2 \frac{|x-y|^2}{r-t_0}}\d r\d x \d\mu_{t_0}(y) \\
    &\leq \kappa_1\ln\Big(\frac{t-t_0}{s-t_0}\Big)  \int_{\R^d} \e^{-\kappa_2|z|^2} \d z.
  \end{align*}
  This yields that
  \[\lim_{t\to s}\|\mu_t-\mu_s\|_{\mathrm{var}}=0,\]
  which implies the weak convergence of $\mu_t$ to $\mu_s$ and further
  \[\lim_{t\to s}\W_1(\mu_t,\mu_s)=0.\]
  The proof is complete.
\end{proof}

\begin{mythm}\label{thm-3.11}
Under the conditions $\mathrm{(H_1)}$-$\mathrm{(H_3)}$, the value function $V(t,\mu)$ given in \eqref{a-2} is a viscosity solution to the HJB equation \eqref{e-1}.
\end{mythm}

\begin{proof}
 \emph{Viscosity subsolution}\ Let $(t_0,\mu_{t_0})\in [0,T)\!\times\! \pb (\bar\oo)$ and $\psi\in C_{L,b}^{1,1}\big([0,T)\!\times\! \pb (\bar\oo)\big)$ be a test function such that
 \[0=(V-\psi)(t_0,\mu_{t_0})=\max\big\{(V-\psi)(t,\mu);\ (t,\mu)\in [0,T)\times \pb (\bar\oo)\big\}. \]
 Let $(X_t,k_t)$ be the solution to SDE \eqref{a-1} associated with the control $\alpha_t\equiv \alpha\in U$ with initial value $\law_{X_{t_0}}=\mu_{t_0}$. Denote $\mu_t=\law_{X_t}$ for $t\geq t_0$.
 By the dynamic programming principle,   we have that
 \[V(t_0,\mu_{t_0})\leq \E \Big[\int_{t_0}^t \vartheta (r,X_r,\mu_r,\alpha)\d r+V(t,\mu_t)\Big],
 \]
 which yields that
 \begin{equation}\label{g-0}
   \psi(t,\mu_t)-\psi(t_0,\mu_{t_0})+\int_{t_0}^t\!\int_{\bar\oo} \vartheta (r,x,\mu_r,\alpha)\d\mu_r(  x)\d r\geq 0.
 \end{equation}
 By Theorem \ref{thm-3.7}, we get
 \begin{align*}
   \int_{t_0}^t\Big[ \partial_r \psi(r,\mu_r)\!&+\!\int_{\bar\oo}\!\Big(\la b(r,x,\mu_r,\alpha), D^L \psi(r, \mu_r)(x)\raa +\frac12\mathrm{tr}\big(A \nabla_x D^L \psi(r,\mu_r)(x)\big)\Big)\d\mu_r(x)\\
   &\qquad   +\int_{\bar\oo}\! \vartheta(r,x,\mu_r,\alpha)\d \mu_r(x)\Big] \d r \geq 0.
 \end{align*}
Using Lemmas \ref{lem-2} and \ref{lem-5},
 dividing both sides of the previous inequality with $t-t_0$, and letting $t\downarrow t_0$, we obtain that
 \begin{align*}
 -\partial_t \psi(t_0,\mu_{t_0})&-\!\int_{\bar\oo}\!\Big(\la b(t_0,x,\mu_{t_0},\alpha), D^L \psi(t_0, \mu_{t_0})(x)\raa +\frac12\mathrm{tr}\big(A \nabla_x D^L \psi(t_0,\mu_{t_0})(x)\big)\Big)\d\mu_{t_0}(x)\\
 &-\!\int_{\bar\oo} \vartheta (t_0,x,\mu_{t_0},\alpha) \d\mu_{t_0}(x)\leq 0.
 \end{align*}
 By the arbitrariness of $\alpha\in U$, we obtain that
 \[-\partial_t\psi(t_0,\mu_0)-\inf_{\alpha\in U} \mathcal{H}(t_0,\mu_{t_0}, \psi,D^L \psi,\alpha) \leq 0.\] Hence, $V$ is a viscosity subsolution to \eqref{e-1}.

 \emph{Viscosity supersolution} \ Let $(t_0,\mu_{t_0})\in [0,T)\!\times \! \pb (\bar\oo)$ and $\psi\in C_{L,b}^{1,1}\big([0,T)\!\times\!\pb (\bar\oo)\big)$ be a test function such that
 \begin{equation}\label{g-1}
 0=(V-\psi)(t_0,\mu_{t_0})=\min\big\{(V-\psi)(t,\mu);\ (t,\mu)\in [0,T)\times \pb (\bar\oo)\big\}.
 \end{equation}
 We shall prove
 \begin{equation}\label{g-2}
 -\partial_t \psi(t_0,\mu_{t_0})-\inf_{\alpha\in U} \mathcal{H}(t_0,\mu_{t_0},\psi,D^L \psi, \alpha) \geq 0
 \end{equation} by contradiction. Suppose
 \begin{equation}\label{g-3}
  -\partial_t \psi(t_0,\mu_{t_0})-\inf_{\alpha\in U} \mathcal{H}(t_0,\mu_{t_0},\psi,D^L \psi, \alpha)<0.
  \end{equation}

 %
%
  For any ${\bm \alpha}\!\in\! \Pi_{t_0,\mu_{t_0}}$, the associated controlled process $(X_t^{t_0,\mu_{t_0}},k_t^{t_0,\mu_{t_0}})_{t\in [t_0,T]}$ is given in \eqref{a-1}. Under the nondegenerate condition $\mathrm{(H_2)}$, the law of $X_t^{t_0,\mu_{t_0}}$ admits a density $\rho_t(x)$ w.r.t. the Lebesgue measure $\d x$ in $\R^d$.  Due to \eqref{rep}, $\rho_t$ admits a representation
  \[\rho_t(x)=\int_{\bar\oo} p(t_0,z; t,x)\d\mu_{t_0}(z) .\]
  %
  Therefore, by Lemmas \ref{lem-2} and \ref{lem-5}, there exist $\veps>0$ and $\zeta_1>0$ such that for any $|t-t_0|<\zeta_1$ and any ${\bm \alpha}\in \Pi_{t_0,\mu_{t_0}}$,
  \begin{equation}\label{g-4}
  \begin{aligned}
     -\partial_t\psi(t,\mu_t)&-\int_{\bar\oo}\!\Big( \la b(t, x,\mu_t,\alpha_t),D^L\psi(t,\mu_t)(x)\raa \!+\!\frac12\mathrm{tr} \big( A\nabla_x D^L\psi(t,\mu_t)(x)\big) \Big)\d \mu_t(x)\\
     &-\int_{\bar\oo}\vartheta(t,x,\mu_t,\alpha_t)\d \mu_t(x) \leq -\veps.
  \end{aligned}
  \end{equation}

  Take two sequence $\delta_n,\,\gamma_n>0$, $n\geq 1$, satisfying
  \[\delta_n< \zeta_1 ,\quad \lim_{n\to \infty} \gamma_n/\delta_n=0.
  \]
  By the dynamic programming principle, there exists a sequence of admissible feedback controls ${\bm \alpha_n}\in\Pi_{t_0,\mu_{t_0}}$ such that
  \[V(t_0,\mu_0)\geq \E\Big[\int_{t_0}^{t_0+\delta_n} \!\vartheta (r,X_r^n,\mu_r^n,\alpha_r^n)\d r+V(t_0+\delta_n,\mu_{t_0+\delta_n}^n)\Big]-\gamma_n,
  \] where $(X_t^n)$ denotes the controlled process associated with ${\bm \alpha_n}$, $\mu_t^n$ denotes the law of $X_t^n$. Due to \eqref{g-1},
  \[\psi(t_0,\mu_0)\geq \E\Big[\int_{t_0}^{t_0+\delta_n}\!\vartheta (r, X_r^n,\mu_r^n,\alpha_r^n)\d r+\psi(t_0+\delta_n, \mu_{t_0+\delta_n}^n)\Big]-\gamma_n.\]
  Hence,
  \begin{align*}
    \frac{\gamma_n}{\delta_n} &\geq \frac1{\delta_n}\E\Big[\int_{t_0}^{t_0+\delta_n}\!\vartheta (r,X_r^n, \mu_r^n, \alpha_r^n)\d r+ \int_{t_0}^{t_0+\delta_n}\!\frac{\d }{\d r}\big(\psi(r,\mu_r^n)\big)\d r\Big].
  \end{align*}
  Since $\psi\in C_{L,b}^{1,1}\big([0,T)\!\times\!\pb(\bar\oo)\big)$, by Theorem \ref{thm-3.7} and \eqref{g-4},
  \begin{align*}
    \frac{\gamma_n}{\delta_n}&\geq \frac{1}{\delta_n}\int_{t_0}^{t_0+\delta_n}\!\!
    \Big[\partial_r\psi(r,\mu_r^n)\!+\!\int_{\bar\oo}\!\Big(\la b(r, x,\mu_r^n,\alpha_r^n), D^L\psi(r,\mu_r^n)(x)\raa \\
    &\qquad \qquad \qquad\quad  +\!\frac 12 \mathrm{tr}\big( A\nabla_xD^L \psi(r,\mu_r^n)(x)\big)\! +\! \vartheta(r,x,\mu_r^n ,\alpha_r^n)\Big)\d\mu_r^n(x) \Big]\d r\\
    &\geq \frac1{\delta_n}\int_{t_0}^{t_0+\delta_n} \veps\d r=\veps>0.
  \end{align*} Letting $n\to \infty$, this contradicts $\lim_{n\to \infty} \frac{\gamma_n}{\delta_n} =0$. Consequently, the assertion \eqref{g-4} is false, and $V(t,\mu)$ is a viscosity supersolution to the HJB equation \eqref{e-1}.
  In all, according to Definition \ref{def-5}, $V$ is a viscosity solution to \eqref{e-1}.
\end{proof}

\section{Comparison Principle for HJB equations}\label{sec-4}

In this part we proceed to study the uniqueness of viscosity solution to the HJB equation \eqref{e-1} associated with the intrinsic derivative.    To this aim, the crucial point is to find suitable test functions to approximate the viscosity solution. In the study of HJB equations on $\R^d$, the Euclidean distance $|x-y|^2$ plays important role in the argument of the comparison principle. On $\pb(\bar\oo)$, although  $\W_2$ is intrinsically differentiable (cf. Proposition \ref{lem-7} below or \cite[Theorem 8.4.7]{Amb}), the $L^2$-Wasserstein distance $\W_2$ is not smooth enough to establish the comparison principle for the HJB equation \eqref{e-1} on $\pb(\bar\oo)$, which will be clarified in the study below. Besides, $(\pb(\bar\oo),\W_2)$ is an infinite dimensional space, bounded sets are not necessary precompact, which is another crucial point to establish the comparison principle for \eqref{e-1}.

The regularity of $\W_2$ w.r.t.\,the intrinsic derivative depends heavily on the theory of optimal transport maps between probability measures, which essentially depends on the study of Monge-Amp\`ere equation. A large number of works have been devoted to the study of the Monge-Amp\`ere equation. We refer to the works of Trudinger-Wang \cite{TW08}, Caffarelli-McCann \cite{CM10}, and Chen et al. \cite{CLW21} amongst others. Let us recall a result in \cite{CLW21} to be used later.

\begin{mythm}[\cite{CLW21}, Theorem 1.1]\label{thm-4}
Suppose $\mathcal{O},\mathcal{O}^\ast$ are bounded convex domains in $\R^d$ with $C^{1,1}$ boundary. Suppose $u$ be a convex solution to the Monge-Amp\`ere equation
\begin{equation}\label{monge}
  \begin{cases}
    \mathrm{det}\big(  D^2 u(x)\big)=\frac{\rho(x)}{\tilde\rho(Du(x))},& x\in \mathcal{O},\\ D u(\mathcal{O})=\mathcal{O}^\ast,
  \end{cases}
\end{equation}
where $\mathrm{det}(B)$ stands for the determinant of matrix $B$. The following assertions hold.
\begin{itemize}
\item[$(\mathrm{i})$] If $\rho\in C^\beta(\bar{\mathcal{O}})$,
$\tilde \rho\in C^\beta(\bar{\mathcal{O}^\ast})$ for some $\beta\in (0,1)$, then
\[\|u\|_{C^{2,\beta}(\bar{\mathcal{O}})}\leq C,\]
where $C$ is a constant depending on $d,\beta,\rho,\tilde\rho,\mathcal{O}$, and $\mathcal{O}^\ast$.
\item[$(\mathrm{ii})$] If $\rho\in C^0(\bar{\mathcal{O}})$, $\tilde\rho\in C^0({\bar{\mathcal{O}^\ast}})$, then
\begin{gather*}
  \|u\|_{C^{1,\beta}(\bar{\mathcal{O}})}\leq C_\beta,\quad \beta\in (0,1);\qquad
  \|u\|_{W^{2,q}(\bar{\mathcal{O}})}\leq C_q,\quad q\geq 1,
\end{gather*}
where the constants $C_\beta$, $C_q$ depend on $d,\rho,\tilde \rho$, $\mathcal{O}$, $\mathcal{O}^\ast$, and on $\beta$, $q$ respectively. $\|\,\cdot\,\|_{W^{2,q}(\bar{\mathcal{O}})}$ denotes the Sobolev norm in the Sobolev space $W^{2,q}(\bar{\mathcal{O}})$.
\end{itemize}
\end{mythm}

Applying the theory on optimal transport maps between probability measures (cf. for example, \cite{Vil}), for two probability measures $\mu=\rho(x)\d x$ and $\nu=\tilde \rho(x)\d x$ on $\R^d$, there exists a convex function $u:\R^d\to\R^d$ such that the mapping $\mathcal{T}^\mu_\nu(x):=D u(x)$ satisfies
\begin{gather*}
  \nu=(\mathcal{T}_\nu^\mu)_\#\mu:=\mu\circ (\mathcal{T}^\mu_\nu)^{-1}, \ \text{i.e.}\ \int h(x)\d\nu(x)=\int h(\mathcal{T}_\nu^\mu(x))\d\mu(x),\quad \forall\,h\in \B_b(\R^d),\\
\W_2^2(\mu,\nu)=\int_{\R^d}\!\!|x-\mathcal{T}_\nu^\mu(x)|^2\d\mu(x).
\end{gather*}
Thus, $u$ is a solution to the Monge-Amp\`ere equation:
\begin{equation*}\mathrm{det}\big( D^2u(x) \big)=\frac{\rho(x)}{\tilde \rho(Du(x))}. \end{equation*}
Moreover, although $u$ is not unique, its gradient and hence the mapping $\mathcal{T}_\nu^\mu$ is unique and invertible. Also, the convexity of $u$ yields that $D^2 u\geq 0$.
Its inverse mapping pushes $\nu$ forward to $\mu$, i.e. $\mu=(\mathcal{T}^{\mu}_{\nu})^{-1}_{\#}\nu=\nu\circ \mathcal{T}_\nu^\mu$. Here we present a result on the optimal transport map for the cost function $c(x,y)=|x-y|^2$, and much effort has been devoted to the study on the general cost functions   and on general spaces such as Riemannian manifolds (cf. \cite{Amb,TW08,Vil}).

\begin{myprop}[Derivative of Wasserstein distance]\label{lem-7}
For each $\zeta\in \pb^r(\bar\oo)$, the associated functional $\mu\mapsto W_2^2(\mu, \zeta)
$ belongs to $C_{L,b}^{1}(\pb^r(\bar\oo))$ and
\begin{equation}\label{h-3}
D^L\W_2^2(\mu,\zeta)(x)=2\big(x-\mathcal{T}^{\mu}_{\zeta}(x)\big),\quad x\in \bar\oo,\, \mu\in \pb^r(\bar\oo),
\end{equation} where $\mathcal{T}_\zeta^\mu:\bar\oo\to \bar\oo$ denotes the unique optimal transport map such that \[\text{$(\mathcal{T}_\zeta^\mu)_{\#}\mu=\zeta$ and }\ \W_2^2(\mu,\zeta) =\int_{\bar\oo}|x-\mathcal{T}_\zeta^\mu(x)|^2\d\mu(x).\]
Furthermore, $x\mapsto D^L\W_2(\mu,\zeta)(x)$ is continuously differentiable. If $\mu,\mu_k\in \pb^r(\bar\oo)$, $k\geq 1$, satisfy $\lim_{k\to \infty }\W_2(\mu_k,\mu)=0$. Then, for any $\veps>0$,
\begin{gather}\label{get-1}
\lim_{k\to \infty}\mu_k\big(\{x\in \bar\oo;\ |\mathcal{T}_\zeta^{\mu_k}(x)-\mathcal{T}_{\zeta}^{\mu}(x)| >\veps\}\big)=0,\\
\label{get-2}
\lim_{k\to \infty} \int_{\bar\oo}\!|\mathcal{T}_{\zeta}^{\mu_k}(x)- \mathcal{T}_{\zeta}^{\mu}(x)|\d \mu_k(x)=0.
\end{gather}
\end{myprop}

\begin{proof}
For $\mu\in \pb^r(\bar\oo)$, and for any tangent vector $v\in \mathscr{T}_\mu$, the curve $\mu_\veps:= \mu\circ(\mathrm{Id}+\veps v)^{-1}$ for $\veps\in [0,1]$ is an absolutely continuous curve in $\pb(\bar\oo)$.
  According to \cite[Theorem 8.4.7]{Amb},
  \begin{equation}\label{h-4}
  \frac{\d}{\d\veps}\Big|_{\veps=0}\!\W_2^2(\mu_\veps,\zeta)= \int_{\bar\oo^2}\!2\la x_1-x_2,v(x_1)\raa\d\gamma(x_1,x_2),
  \end{equation} where $\gamma$ is a probability measure on $\bar\oo\times\bar\oo$ satisfying
  \[\int_{\bar\oo^2}\!f(x_1)+g(x_2)\d\gamma(x_1,x_2) =\int_{\bar\oo} f(x_1)\d \mu(x_1)+\int_{\bar\oo} \!g(x_2)\d\zeta(x_2),\quad f,g\in \mathcal{B}_b(\bar\oo),\]
  and
  \[\int_{\bar\oo^2}\!|x_1-x_2|^2\d\gamma(x_1,x_2) =\W_2^2(\mu,\zeta).\]
  By virtue of the results on optimal transport maps (cf. e.g. \cite[Chapter 6]{Amb}), since
  $\mu\in \pb^r(\bar\oo)$ admits a density w.r.t.\,the Lebesgue measure, the previous optimal plan $\gamma$ is uniquely determined by
  \[\gamma=(\mathrm{Id}\times \mathcal{T}_\zeta^\mu)_{\#}\mu. \]
  Moreover, there exists a function $u:\bar\oo\to \R$ satisfying the Monge-Amp\`ere equation \eqref{monge} with $\bar\oo^\ast=\bar\oo$, $\rho=\d\mu/\d x$ and $\tilde \rho=\d\zeta/\d x$ such that
  \[\mathcal{T}_\zeta^\mu(x)=D u(x).\]
  By Theorem \ref{thm-4}, as $\rho,\tilde \rho\in C^1(\bar\oo)$, for each $\beta\in (0,1)$, there exists a constant $C>0$ such that
  \[\|u\|_{C^{2,\beta}(\bar{\oo})}\leq C.\]
  This yields that $\mathcal{T}_\zeta^\mu=D u$ is in $C^{1,\beta}(\bar{\oo})$. Consequently, we can rewrite \eqref{h-4} to
  \[\frac{\d}{\d \veps}\Big|_{\veps=0}\W_2^2(\mu_\veps,\zeta)=\int_{\bar\oo}2\la x_1-\mathcal{T}_\zeta^\mu(x_1),v(x_1)\raa \d\mu(x_1), \quad v\in \mathscr{T}_\mu.\]
  This yields that $\W^2_2(\cdot,\zeta)$ is intrinsically differentiable at $\mu$ with  \[D^L\W_2^2(\mu,\zeta)(x)=2(x-\mathcal{T}^\mu_{\zeta}(x)).\]
   Moreover, as $u\in C^{2,\beta}(\bar\oo)$, $x\mapsto D^L\W_2^2(\mu,\zeta)(x)=2(x-Du(x))$ is continuous in $x\in \bar{\oo}$. Obviously, $D^L\W_2^2(\mu,\zeta)$ is also bounded as $\bar\oo$ is compact.

 Next, we proceed to consider the continuity of $\mu\mapsto D^L \W_2^2(\mu,\zeta)(x)$ from $\pb^r(\bar\oo)$ to $L^1(\bar\oo)$.  This is useful to study the stability of the optimal transport maps between probability measures.

  Let  $\mu_k,\mu \in \pb^r(\bar\oo)$, $k\geq 1$, satisfying $\lim_{k\to \infty}\W_2(\mu_k,\mu)= 0$. For any $\veps>0$ define
  \[B_\veps=\big\{(x,y)\in \bar{\oo}\times \bar{\oo};\, |\mathcal{T}_\zeta^{\mu}(x)-y|\geq \veps\big\}.\]
  Put $\gamma_k=(\mathrm{Id}, \mathcal{T}_\zeta^{\mu_k})_{\#}\mu_k\in \mathscr{C}(\mu_k,\zeta)$ is the optimal transport plan of $\mu_k$ and $\zeta$ for $k\geq 1$, and $\gamma=(\mathrm{Id},\mathcal{T}_\zeta^{\mu})_{\#}\mu\in \mathscr{C}(\mu,\zeta)$, where $\mathrm{Id}$ denotes the identity map. Since $\mu$ is absolutely w.r.t.\, the Lebesgue measure, $\gamma$ is the unique optimal transport plan in $\mathscr{C}(\mu,\zeta)$ relative to $\W_2$ distance (cf. \cite{Vil}). Due to Theorem \ref{thm-4}, $x\mapsto \mathcal{T}_{\zeta}^{\mu}(x)=Du(x)$ is continuous. Together with the compactness of $\bar{\oo}$, we get that $B_\veps$ is closed.

  By virtue of the weak convergence of $\mu_k$  to $\mu$ and the uniqueness of optimal transport map $\gamma$, we have $\gamma_k$ converges weakly to $\gamma\in \mathscr{C}(\mu,\zeta)$ as $k\to \infty$ by \cite[Theorem 5.20]{Vil}. Hence,
  \begin{align*}
    0=\gamma(B_\veps)&\geq \limsup_{k\to \infty}\gamma_k(B_\veps)\\
    &=\limsup_{k\to \infty} \gamma_k\big(\{(x,y)\in \bar{\oo}\times \bar{\oo}; |\mathcal{T}^{\mu}_{\zeta}(x)-y|\geq \veps\}\big)\\
    &=\limsup_{k\to \infty}\mu_k\big(\{x\in \bar{\oo}; |\mathcal{T}^{\mu}_{\zeta}(x)- \mathcal{T}^{\mu_k}_{\zeta}(x)|\geq \veps\}\big).
  \end{align*}
  Therefore, this implies immediately \eqref{get-1}.

  Furthermore,
  \begin{align*}
    &\lim_{k\to \infty} \int_{\bar\oo}\big|\mathcal{T}^{\mu_k}_\zeta(x)- \mathcal{T}^\mu_\zeta(x)\big|\d \mu_k(x)\\
    &\leq \veps+\lim_{k\to \infty} \int_{\big|\mathcal{T}^{\mu_k}_\zeta(x)- \mathcal{T}^\mu_\zeta(x)\big|\geq \veps} \big|\mathcal{T}^{\mu_k}_\zeta(x)- \mathcal{T}^\mu_\zeta(x)\big|\d\mu_k(x)\\
    &\leq \veps+\mathrm{diam}(\bar\oo)\limsup_{k\to \infty} \mu_k\big(\big \{x\in \bar\oo;\ \big|\mathcal{T}^{\mu_k}_\zeta(x)- \mathcal{T}^\mu_\zeta(x)\big|\geq \veps\big\}\big)\\
    &\leq \veps,
  \end{align*} where $\mathrm{diam}(\bar\oo)=\sup\{|x-y|;\,x,y\!\in\! \bar\oo\}<\!\infty$ since $\bar\oo$ is compact.
  The arbitrariness of $\veps$ implies
  \[\lim_{k\to \infty} \int_{\bar\oo}\big|\mathcal{T}^{\mu_k}_\zeta(x)- \mathcal{T}^\mu_\zeta(x)\big|\d \mu_k(x)=0.\]
  Then \eqref{get-2} holds and
  the proof is complete.
\end{proof}

\begin{myrem}\label{rem-1}
From Proposition \ref{lem-7} we can see that the square of Wasserstein distance $ \mu\mapsto \W_2^2(\mu,\zeta)$ is intrinsically differentiable on the subset $\pb^r(\bar\oo)$. To ensure the existence of second order differentiability of $x\mapsto D^L \W_2^2(\mu,\zeta)(x)$, further smoothness condition on the densities of $\mu$ and $\zeta$ is needed, which is needed to act as a smooth approximation function to the HJB equation \eqref{e-1}.  However, the completion of $\pb^r(\bar\oo)$ under the metric $\W_2$ will be $\pb(\bar\oo)$, which cannot guarantee the desired smoothness of   densities.  Thus, $\W_2^2$ is not an appropriate smooth approximation function to study the uniqueness of viscosity solution to the HJB equation \eqref{e-1}.
\end{myrem}

In the recent work, Burzoni et al. \cite{BIRS} studied the optimal control problem for McKean-Vlasov jump-diffusion processes, and developed the theory of viscosity solution to HJB equations on Wasserstein space in terms of linear functional derivative. To establish the comparison principle, they constructed a distance-like function for two probability measures by
\[d(\mu,\nu)=\sum_{j=1}^\infty c_j\la \mu-\nu, f_j\raa^2,\]
where the countable set $\{f_j\}_{j\in \N}$ is carefully constructed so that the linear functional derivative of $d$ can be estimated by itself. The construction of $\{f_j\}_{j\in \N}$ is very subtle especially in the presence of jumps in the controlled process.
\cite{BIRS}  only constructed $d(\mu,\nu)$ on the Wasserstein space  $\pb(\R)$ over the real line $\R$.

In this work we shall generalize the construction of $d(\mu,\nu)$ in \cite{BIRS} to the Wasserstein space over  $\oo\!\subset\!\R^d$. Moreover, we shall show that such kind of distance-like function is also useful to establish the comparison principle for HJB equations on the Wasserstein space with intrinsically differential structure. Unfortunately,  another  assumption on the drift $b$ is  needed like in \cite{BIRS}, that is,  the drift $b(t,x,\mu,\alpha)$ cannot not depend on variable $x$ and depends on $\mu$ via its moments.

Before establishing the comparison principle, we  introduce the generalization of the distance-like function $d(\mu,\nu)$ on $\pb(\bar\oo)$. Let us begin with the 1-dimensional case by recalling the construction in \cite{BIRS}. Since our controlled processes are diffusion processes without jumps, we can simplify the expression $\{f_j\}_{j\in \N}$ in \cite{BIRS}.

\begin{mydefn}\label{defn-4.1}
A set of polynomials $\chi$ is said to have $(\ast)$-property if it satisfies that
\[\text{for any $f\in \chi$, $f^{(i)}\in \chi$, $\forall\,i\geq 0$},\]
where $f^{(i)}$ denotes $i$th order derivative of $f$ with $f^{(0)}=f$.
\end{mydefn}
For any given polynomial $f$, let $\chi(f)$ be the smallest set of polynomials with $(\ast)$-property that includes $f$. So,
\[\chi(f)=\big\{f^{(i)}; i\geq 0\big\},\]
and $\chi(f)$ contains finite number of polynomials for any polynomial $f$.
Put
\[\wt{\Theta}=\bigcup_{j=1}^\infty \chi(x^j).\]
Then $\wt{\Theta}$ contains all monomials $\{x^j\}_{j=1}^\infty$, it is countable and $\chi(f)\subset \wt{\Theta}$ for every $f\in \wt{\Theta}$. Let $\{f_j\}_{j=1}^\infty$ be an enumeration of $\wt{\Theta}$, which is fixed in the sequel. We refer to \cite{BIRS} for more discussion on $\chi,\,\chi(f)$ and $(\ast)$-property.

Now we generalize the previous notions to multidimensional situation.
Denote $\Z_+=\{0,1,2,\ldots\}$, and ${\bm x}^{\bm{n}}=x_1^{n_1}x_2^{n_2}\cdots x_d^{n_d}$ for $\bm{n}=(n_1,\ldots,n_d)\!\in \!\Z_+^d$, $\bm{x}=(x_1,\ldots,x_d)\!\in \!\R^d$. Let $|\bm{n}|=n_1+ \cdots+n_d$. For $\mu\in \pb(\bar\oo)$ and $f:\bar\oo\to \R^k$, denote by $\la \mu, f\raa=\int_{\bar\oo} f(x)\d \mu(x)$. Based on the above fixed enumeration $\{f_j\}_{j=1}^\infty$ of $\wt \Theta$, we define
\[f_{\bm{n}}(\bm x)=f_{n_1}(x_1)f_{n_2}(x_2)\cdots f_{n_d}(x_d),\quad \text{for}\ \bm{n}=(n_1,\ldots,n_d),\]
and
\[s_{\bm{n}}=1+\sup_{\mu\in\pb(\bar\oo)}\la \mu, f_{\bm n}\raa^2.\]
Put
\[\chi(f_{\bm{n}})=\big\{g_1(x_1)g_2(x_2)\cdots g_d(x_d); g_i\in \chi(f_{n_i}),1\leq i\leq d\big\},\]
which contains all the partial derivatives of $f_{\bm{n}}$. Since $\chi(f_{\bm{n}})$ contains a finite number of polynomials, there exists a finite index set $\mathcal{I}_{\bm{n}}$ satisfying
\[\chi(f_{\bm{n}})=\big\{f_{\bm{m}};\ \bm{m}\in \mathcal{I}_{\bm n}\big\}.
\]
Let
\[c_{\bm{n}}=\Big(\sum_{\bm{m}\in \mathcal{I}_{\bm n}} 2^{|\bm m|}C_{d-1}^{|\bm m|+d-1}\Big)^{-1} \Big( \sum_{\bm m\in \mathcal{I}_{\bm n}} s_{\bm m}\Big)^{-1}, \ \bm{n}\in \Z_+^d,\]
where $C_{k}^m=\frac{m!}{k! (m-k)!}$. As $s_{\bm m}\geq 1$ and $f_{\bm n}\in \chi(f_{\bm n})$, it holds
\[ c_{\bm n} \leq \frac{2^{-|\bm{n}|}}{C_{d-1}^{|\bm{n}|+d-1}}.\]

If $\bm{l}\in \mathcal{I}_{\bm n}$, then $f_{\bm l}\in \chi(f_{\bm n})$,
\[\chi(f_{\bm l})\subset \chi(f_{\bm n}),\]
and hence $\mathcal{I}_{\bm l}\subset \mathcal{I}_{\bm n}$. By the definition of $c_{\bm n}$, this implies that
\begin{equation}\label{h-0}
c_{\bm l}\geq c_{\bm n}.
\end{equation}

\begin{mylem}\label{lem-4}
Define a function $S:\pb(\bar\oo)\to \R$ by
\begin{equation}\label{h-1} S(\mu)=\sum_{k=1}^\infty \sum_{|\bm n|=k} c_{\bm n} \la \mu,f_{\bm n}\raa^2.
\end{equation} Then $S$ satisfies $S(\mu)\leq 1$ for every $\mu\in \pb(\bar\oo)$, and  is intrinsically differentiable with
\begin{equation}\label{h-2}
D^L S(\mu)(x)=\sum_{k=1}^\infty\sum_{|\bm{n}|=k}2 c_{\bm n} \la \mu,f_{\bm n}\raa \nabla f_{\bm n}(x).
\end{equation}
\end{mylem}

\begin{proof}
  Noting that ${\bm n}\in \mathcal{I}_{\bm n}$, we have
  \[S(\mu)\leq \sum_{k=1}^\infty \sum_{|\bm{n}|=k} 2^{-k}\big(C_{d-1}^{k+d-1}\big)^{-1}\frac{\la \mu, f_{\bm n}\raa^2}{s_{\bm n}}\leq \sum_{k=1}^{\infty} 2^{-k}=1.\]
  For any $v\in \mathscr{T}_{\mu}$,
  \begin{align*}
    D^L_v S(\mu)&=\lim_{\veps\to 0}\frac{ S(\mu\circ(\mathrm{Id}+\veps v)^{-1})-S(\mu)}{\veps}\\
    &=\frac{\d }{\d\veps}\Big|_{\veps=0} \sum_{k=1}^{\infty}\sum_{|\bm{n}|=k} c_{\bm n} \Big(\int_{\oo} f_{\bm n}(x+\veps v(x))\d\mu(x)\Big)^2\\
    &=\sum_{k=1}^\infty \sum_{|\bm{n}|=k} 2c_{\bm n} \la \mu, f_{\bm n}\raa \int_{\oo}\!\la \nabla f_{\bm n}(x), v(x)\raa \d\mu(x).
  \end{align*}
  Therefore, $S$ is intrinsically differentiable with $D^L S(\mu)$ given by \eqref{h-2}.
\end{proof}

We need to modify the condition  satisfied by the drift $b$ to  establish the comparison principle.
\begin{itemize}
  \item[$(\mathrm{H_1'})$] The drift $b(t,x,\mu,\alpha)$ does not depend on $x$.  There are $K_4\!>0$ and a finite set $\mathcal{I}\subset \N^d$ such that  for any $\mu,\,\nu\in \pb(\bar\oo)$, $\alpha\in  U$, $t,\,s\in [0,T]$,
      \[|b(t,\mu,\alpha)|^2\leq K_4,\quad |b(t,\mu,\alpha)-b(s,\nu,\alpha)|^2\leq K_4\big(|t-s|^2+\sum_{\bm{i}\in \mathcal{I}} \la \mu-\nu,\bm{x}^{\bm i}\raa^2\big),
      \]
      where ${\bm i}=(i_1,\ldots,i_d)$, ${\bm x}=(x_1,\ldots,x_d)$.
\end{itemize}
This condition is especially suitable to deal with SDEs with drifts depending only on the moments of  $(X_t)$, for example,
\[\d X_t=b(t, \E[X_t],\E[|X_t|^2],\alpha_t)\d t+\sigma\d B_t-\mathbf{n}(X_t)\d k_t.\]

\begin{mythm}[Comparison principle]\label{thm-5}
Suppose that $\mathrm{(H_1')}$, $\mathrm{(H_2)}$ and $\mathrm{(H_3)}$ hold.
Let $W$ and $V$ be respectively a viscosity subsolution and a viscosity supersolution to the HJB equation \eqref{e-1} satisfying the continuity property \eqref{a-4}. Then
\begin{equation}\label{h-5}
W(t,\mu)\leq V(t,\mu),\qquad t\in [0,T),\ \mu\in \pb(\bar\oo).
\end{equation}
\end{mythm}

\begin{proof}
  We shall prove \eqref{h-5} by contradiction. Suppose there exists a point
  $(\tilde{t},\bar{\mu})\in [0,T)\times \pb(\bar\oo)$ such that
  \begin{equation}\label{h-5.5}
  W(\tilde t,\bar \mu)>V(\tilde t,\bar \mu).
  \end{equation}
  Then, by the continuity of $W$ and $V$,  there exist some   $(\bar t,\bar\mu)\in (0,T)\times \pb(\bar\oo)$ such that
  \begin{equation}\label{h-6}
  W(\bar{t},\bar\mu)>V(\bar t,\bar\mu).
  \end{equation}

  Consider the auxiliary function
  \begin{equation*}
    \Phi(t,s,\mu,\nu)=W(t,\mu)-V(s,\nu)-\varphi(t,s,\mu,\nu), \quad t,s\in (0,T],\ \mu,\nu\in \pb(\bar\oo),
  \end{equation*}
  with
  \begin{equation}\label{phi}
  \varphi(t,s, \mu,\nu)=\frac 1{2\delta}\big(|t-s|^2+ S(\mu-\nu) \big)+\beta (2T-t-s)+\frac{\lambda}{t}+\frac{\lambda}{s}
  \end{equation} for parameters $\beta,\lambda,\delta\in (0,1)$, where the functional $S$ is defined in Lemma \ref{lem-4}.
  Due to the compactness of $[0,T]\!\times\![0,T]\!\times\! \pb(\bar\oo)\!\times\!\pb(\bar\oo)$ when $\pb(\bar\oo)$ is endowed with weak convergence topology, there exists a point $(t_0,s_0,\mu_0,\nu_0)\in (0,T]\!\times\!(0,T]\!\times\! \pb(\bar\oo)\!\times\!\pb(\bar\oo)$ such that
  \begin{equation}\label{h-7}
  \Phi(t_0^\tau,s_0^\tau,\mu_0^\tau,\nu_0^\tau)=\sup \big\{\Phi(t,s,\mu,\nu);t,s\in [0,T], \mu,\nu\in\pb(\bar\oo)\big\}.
  \end{equation}
  Notice that $(t_0^\tau,s_0^\tau,\mu_0^\tau,\nu_0^\tau)$  depend on the parameters $\tau=(\beta,\lambda,\delta)$. By \eqref{h-7},
  \[2\Phi(t_0^\tau,s_0^\tau,\mu_0^\tau,\nu_0^\tau)\geq \Phi(t_0^\tau,t_0^\tau,\mu_0^\tau,\mu_0^\tau) +\Phi(s_0^\tau,s_0^\tau,\nu_0^\tau,\nu_0^\tau),
  \] which yields that
  \begin{equation}\label{h-7.5}
  W(t_0^\tau,\mu_0^\tau) -W(s_0^\tau,\nu_0^\tau) +V(t_0^\tau,\mu_0^\tau)- V(s_0^\tau,\nu_0^\tau)\geq \frac{1}{\delta}\big(|t_0^\tau-s_0^\tau|^2+ S(\mu_0^\tau -\nu_0^\tau)\big).
  \end{equation}
  Since $W$ and $V$ are bounded which follows from the boundedness of $\vartheta$ and $g$, this implies that
  \begin{equation}\label{h-8}
    \lim_{\delta\downarrow 0} |t_0^\tau-s_0^\tau |^2+ S(\mu_0^\tau-\nu_0^\tau)=0.
  \end{equation}

  Due to the compactness of $[0,T]\times[0,T]\times \pb(\bar\oo)\times \pb(\bar\oo)$,  there is a subsequence $(t_0^{\tau_k},s_0^{\tau_k},\mu_0^{\tau_k},\nu_0^{\tau_k})$ of $(t_0^\tau,s_0^{\tau}, \mu_0^\tau,\nu_0^\tau)$ satisfying  $\tau_k=(\beta,\lambda_k,\delta_k)\longrightarrow (\beta,0,0)$ as $k\to \infty$, and
  \begin{equation}\label{h-8.5} (t_0^{\tau_k},s_0^{\tau_k},\mu_0^{\tau_k},\nu_0^{\tau_k})\ \text{converges to some } (\bar t_0,\bar s_0,\bar\mu_0,\bar\nu_0).
  \end{equation}
  Due to \eqref{h-8}, it holds
  \[\bar t_0=\bar s_0,\quad S(\bar\mu_0-\bar\nu_0)=\lim_{k\to \infty} S(\mu_0^{\tau_k}-\nu_0^{\tau_k})=0.\]
  Since the class of polynomials $\{\bm{x}^{\bm n}\}_{\bm n\in \Z_+^d}\subset\{f_{\bm n}\}_{\bm{n}\in \Z_+^d}$ is a   measure determining class of functions, then we get from $S(\bar\mu_0-\bar \nu_0)=0$ that  $\bar\mu_0=\bar\nu_0$. Combining this with \eqref{h-8.5}, \eqref{a-4} and \eqref{h-7.5}, we obtain that
  \begin{equation}\label{h-8.6}
  \lim_{k\to \infty}\frac{1}{\delta_k}\big(|t_0^{\tau_k}-s_0^{\tau_k}|^2+ S(\mu_0^{\tau_k}-\nu_0^{\tau_k})\big)=0.
  \end{equation}

  Case 1: If for some sequence $\tau_k=(\beta,\lambda_k,\delta_k)$, the corresponding maximum points $(t_0^{\tau_k},s_0^{\tau_k}, \mu_0^{\tau_k},\nu_0^{\tau_k})$ satisfy $t_0^{\tau_k}\vee s_0^{\tau_k}=T$. By \eqref{h-7},
  \[\Phi(\bar t,\bar t, \bar\mu,\bar\mu)\leq \Phi(t_0^{\tau_k}, s_0^{\tau_k}, \mu_0^{\tau_k},\nu_0^{\tau_k}),
  \] which yields
  \begin{align*}
    &W(\bar t,\bar \mu)\!-\!V(\bar t,\bar\mu)\!-\! 2\beta(T-\bar t)\! -\! \frac{2\lambda_k}{\bar t}\!\leq W(t_0^{\tau_k},\mu_0^{\tau_k})\!- \! V(s_0^{\tau_k},\nu_0^{\tau_k}) \!-\!\frac1{2\delta_k}\big( |t_0^{\tau_k}\!-\! s_0^{\tau_k}|^2\! +\! S(\mu_0^{\tau_k}\!-\! \nu_0^{\tau_k})\big).
  \end{align*}
  Letting $k\to \infty$, due to \eqref{h-8.6} and $W(T,\mu)=V(T,\mu)$ for any $\mu\in\pb(\bar\oo)$,
  \[W(\bar t,\bar\mu)-V(\bar t,\bar\mu)-\beta(T-\bar t)\leq 0.\]
  Then, letting $\beta\to 0$, we get
  $W(\bar t,\bar \mu)-V(\bar t,\bar \mu)\leq 0$, which contradicts \eqref{h-6}.

  Case 2: For any $\tau_k=(\beta,\lambda_k,\delta_k)$, the corresponding maximum points satisfy $t_0^{\tau_k}\vee s_0^{\tau_k}<T$.  Let
  \begin{equation*}
  \psi(t,\mu)=V(s_0^{\tau_k},\nu_0^{\tau_k})+ \frac{1}{2\delta_k}\big( |t-s_0^{\tau_k}|^2\!+\!  S(\mu-\nu_0^{\tau_k})\big)  +\beta (2T-t-s_0^{\tau_k}) +\frac{\lambda_k}{t}+\frac{\lambda_k}{s_0^{\tau_k}}.
  \end{equation*} According to Lemma \ref{lem-4}, $\psi\in C_{L,b}^{1,1}\big((0,T]\times \pb(\bar\oo)\big)$.
  Consider the function
  \[(t,\mu)\mapsto W(t,\mu)-\psi(t,\mu)=\Phi(t,s_0^{\tau_k},\mu,\nu_0^{\tau_k})\]
  which attains its maximum at $(t_0^{\tau_k},\mu_0^{\tau_k})$ by \eqref{h-7}. Because $W$ is a viscosity subsolution to \eqref{e-1}, it holds
  \begin{equation}\label{h-9}
  -\partial_t\psi(t_0^{\tau_k},\mu_0^{\tau_k}) -\inf_{\alpha\in  U }\mathcal{H}\big(t_0^{\tau_k},\mu_0^{\tau_k}, \psi, D^L\psi,\alpha\big) \leq 0.
  \end{equation}
  Analogously, let
  \[\tilde \psi(s,\nu)=W(t_0^{\tau_k},\mu_0^{\tau_k})\!-\! \frac1{2\delta_k} \big( |t_0^{\tau_k}\!-s|^2\!+\! S(\mu_0^{\tau_k}\!-\nu)  \!-\!\beta(2T-t_0^{\tau_k}\!-s)\!-\! \frac{\lambda_k}{t_0^{\tau_k}} \!-\!\frac{\lambda_k}{s}.\]
  Then, $(s,\nu)\mapsto V(s,\nu)-\tilde \psi(s,\nu)$ attains its minimum at $(s_0^{\tau_k},\nu_0^{\tau_k})$. As $V$ is a viscosity supersolution to \eqref{e-1}, it holds
  \begin{equation}\label{h-10}
  -\partial_s\tilde \psi(s_0^{\tau_k},\nu_0^{\tau_k})- \inf_{\alpha\in  U }\mathcal{H}\big( s_0^{\tau_k}, \nu_0^{\tau_k}, \tilde \psi, D^L\tilde\psi, \alpha\big) \geq 0.
  \end{equation}
  Combining \eqref{h-10} with \eqref{h-9}, we obtain
  \begin{equation}\label{h-10.5}
  \begin{aligned}
    &\partial_t\psi(t_0^{\tau_k},\mu_0^{\tau_k})- \partial_s\tilde\psi(s_0^{\tau_k},\nu_0^{\tau_k})\\ &\geq \inf_{\alpha\in U}\mathcal{H}\big( s_0^{\tau_k}, \nu_0^{\tau_k}, \tilde \psi, D^L\tilde\psi, \alpha\big)
    -\inf_{\alpha\in U } \mathcal{H}\big(t_0^{\tau_k},\mu_0^{\tau_k}, \psi, D^L\psi,\alpha\big) \\
    &\geq \inf_{\alpha\in U }\Big\{ \mathcal{H}\big( s_0^{\tau_k}, \nu_0^{\tau_k}, \tilde \psi, D^L\tilde\psi, \alpha\big) - \mathcal{H}\big(t_0^{\tau_k},\mu_0^{\tau_k}, \psi, D^L\psi,\alpha\big)\Big\}.
  \end{aligned}
  \end{equation}
  According to Lemma \ref{lem-4},
  \begin{align*}
    D^L\psi(t_0^{\tau_k},\mu_0^{\tau_k})(x) = D^L\tilde\psi(s_0^{\tau_k},\nu_0^{\tau_k})= \frac{1}{\delta_k}\sum_{m=1}^\infty\sum_{|\bm{n}|=m}c_{\bm n} \la \mu_0^{\tau_k}-\nu_0^{\tau_k}, f_{\bm n}\raa \nabla f_{\bm n}(x).
  \end{align*}
  By direct calculation, we get from \eqref{h-10.5} that
  \begin{equation}\label{h-11}
  \begin{split}
&-2 \beta   \\
&\geq \! \inf_{\alpha\in  U }\! \Big\{\!\int_{\bar{\oo}} \la b(s_0^{\tau_k} , \nu_0^{\tau_k} ,\alpha), D^L\tilde \psi(s_0^{\tau_k},\nu_0^{\tau_k} )(x)\raa\d \nu_0^{\tau_k} (x)\\
   &\qquad   \qquad -\!\int_{\bar{\oo }} \la b(t_0^{\tau_k}, \mu_0^{\tau_k},\alpha),D^L\psi (t_0^{\tau_k},\mu_0^{\tau_k})(x)\raa\d \mu_0^{\tau_k} (x)\\
   &\qquad  +\!\frac 12\Big(\!\int_{\bar{\oo}}\! \mathrm{tr}\big( A  \nabla_x  D^L\tilde  \psi(s_0^{\tau_k},\nu_0^{\tau_k} )(x)\big) \d \nu_0^{\tau_k}(x) \!-\!  \int_{\bar\oo} \! \mathrm{tr}\big( A \nabla_x D^L
     \psi(t_0^{\tau_k},\mu_0^{\tau_k})(x)\big) \d \mu_0^{\tau_k}(x)\Big) \\
   &\qquad   +\int_{\bar{\oo}}\vartheta (s_0^{\tau_k} ,x,\nu_0^{\tau_k} ,\alpha)\d \nu_0^{\tau_k} (x)-\int_{\bar{\oo}}\vartheta (t_0^{\tau_k} ,x,\mu_0^{\tau_k} ,\alpha)\d \mu_0^{\tau_k} (x)\Big\}\\
   &=:\inf_{\alpha\in U }\big\{\mathrm{(I)}+\mathrm{(I\!I)}+ \mathrm{(I\!I\!I)}\big\}.
  \end{split}
  \end{equation}
  We shall estimate these three terms one by one.

  First, let us estimate term $\mathrm{(I)}$.  By $\mathrm{(H_1')}$, there exists $K>0$ such that
  \[|b(t,\mu,\alpha)-b(s,\nu,\alpha)|^2\leq K \big( |t-s|^2+ S(\mu-\nu)\big),\quad t,s\in [0,T],\ \mu,\nu\in \pb(\oo),\]
  where the functional $S$ is given in Lemma \ref{lem-4}.
  \begin{align*}
    \mathrm{(I)}&=\frac{1}{\delta_k} \sum_{m=1}^\infty\sum_{|\bm{n}|=m}c_{\bm n} \la \mu_0^{\tau_k}-\nu_0^{\tau_k}, f_{\bm n}\raa \Big( \int_{\bar\oo}\la b(s_0^{\tau_k},\nu_0^{\tau_k},\alpha) -b(t_0^{\tau_k}, \mu_0^{\tau_k},\alpha),\nabla f_{\bm n}(x)\raa \d \nu_0^{\tau_k}(x)\\
    &\qquad\qquad \qquad\qquad\qquad \qquad\qquad  \quad +\int_{\bar\oo}\la b(t_0^{\tau_k},\mu_0^{\tau_k}, \alpha), \nabla f_{\bm n}(x)\raa \d (\nu_0^{\tau_k}-\mu_0^{\tau_k})(x)\Big)\\
    &\geq -\frac{1}{2\delta_k}\!\sum_{m=1}^\infty\!\sum_{|\bm{n}|=m} \!\! c_{\bm n}\Big[ \la \mu_0^{\tau_k}\!-\!\nu_0^{\tau_k},f_{\bm n}\raa^2 \!+\! 2 \Big(\la b(s_0^{\tau_k},\nu_0^{\tau_k},\alpha)\!-\!b(t_0^{\tau_k}, \mu_0^{\tau_k},\alpha),\int_{\bar\oo}\!\!\nabla f_{\bm n}(x)\d \nu_0^{\tau_k}(x)\raa \Big)^2\\
    &\qquad \qquad \qquad \qquad \quad + 2\Big(\la b(t_0^{\tau_k},\mu_0^{\tau_k},\alpha),\int_{\bar\oo} \nabla f_{\bm n}(x)\d (\nu_0^{\tau_k}-\mu_0^{\tau_k})(x)\Big)^2 \Big]\\
    &\geq -\frac{1}{2\delta_k} \! \sum_{m=1}^\infty\!\sum_{|\bm{n}|=m} \!c_{\bm n}\!\Big[ \la \mu_0^{\tau_k}\!-\!\nu_0^{\tau_k},f_{\bm n}\raa ^2\!+\! 2|b(s_0^{\tau_k},\nu_0^{\tau_k},\alpha)\!-\!b(t_0^{\tau_k}, \mu_0^{\tau_k},\alpha)|^2 \Big|\int_{\bar\oo}\!\!\nabla  f_{\bm n}(x)\d \nu_0^{\tau_k}(x)\Big|^2\\
    &\qquad \qquad \qquad \qquad +2|b(t_0^{\tau_k},\mu_0^{\tau_k}, \alpha)|^2 \Big|\int_{\bar\oo}\!\nabla f_{\bm n}(x)\d  (\nu_0^{\tau_k} -\mu_0^{\tau_k})(x)\Big|^2\Big]\\
    &\geq -\frac{1}{2\delta_k}\sum_{m=1}^\infty\!\sum_{|\bm{n}|=m}\! c_{\bm n}\Big[ \la \mu_0^{\tau_k}-\nu_0^{\tau_k}, f_{\bm n}\raa^2 \!+\! 2K\big(|t_0^{\tau_k}-s_0^{\tau_k}|^2\!+\! S(\nu_0^{\tau_k}-\mu_0^{\tau_k})\big) \Big| \int_{\bar \oo}\! \nabla f_{\bm n}(x)\d \nu_0^{\tau_k}(x)\Big|^2\\
    &\qquad \qquad \qquad \qquad + 2K_4  \Big|\int_{\bar \oo}\!\nabla f_{\bm n}(x) \d (\nu_0^{\tau_k}-\mu_0^{\tau_k})(x)\Big|^2\Big].
  \end{align*}
Notice that $\partial_{x_i}f_{\bm n}(x)\in \chi(f_{\bm n})$ for each $i\in\! \{1,\ldots,d\}$, and so there exists $\bm{l}_i(\bm{n})\in \mathcal{I}_{\bm n}$ such that
$ \partial_{x_i} f_{\bm n}(x)=f_{\bm{l}_i(\bm{n})}(x)$. By \eqref{h-0}, $c_{\bm{l}_i(\bm{n})}\geq c_{\bm{n}}$, and
\[c_{\bm n}\Big| \int_{\bar\oo} \partial_{x_i} f_{\bm n}(x) \d \nu_0^{\tau_k}(x)\Big|^2\leq c_{\bm{l}_i(\bm{n})}\Big|\int_{\bar \oo}f_{\bm{l}_i(\bm{n})}(x)\d \nu_0^{\tau_k}(x)\Big|^2\]
Besides, Lemma \ref{lem-4} tells us that $S(\nu_0^{\tau_k})\leq 1$.
This yields that
\begin{align*}
\sum_{m=1}\sum_{|\bm{n}|=m}c_{\bm n} \Big|\int_{\bar\oo}\! \nabla f_{\bm n}(x)\d \nu_0^{\tau_k}(x)\Big|^2& \leq\sum_{m=1}^\infty \sum_{|\bm{n}|=m}\! \sum_{i=1}^d c_{\bm n} \Big|\int_{\bar\oo} \partial_{x_i} f_{\bm{n}}(x)\d \nu_0^{\tau_k}(x)\Big|^2   \\
  &\leq \sum_{m=1}^\infty \sum_{|\bm{n}|=m}\! \sum_{i=1}^d c_{\bm{l}_i(\bm{n})} \la \nu_0^{\tau_k}, f_{\bm{l}_i(\bm{n})}\raa^2 \leq d,
\end{align*}
and
\begin{equation}\label{h-12}
\begin{aligned}
  &\sum_{m=1}\sum_{|\bm{n}|=m} c_{\bm n} \Big|\int_{\bar\oo} \! \nabla f_{\bm n} (x)\d (\nu_0^{\tau_k}-\mu_0^{\tau_k})(x)\Big|^2
  \leq   d \sum_{m=1}\sum_{|\bm{n}|=m} c_{\bm n} \la \nu_0^{\tau_k} -\mu_0^{\tau_k}, f_{\bm n}\raa^2.
\end{aligned}
\end{equation}
Invoking the previous estimates, we finally obtain that
\begin{equation}\label{h-13}
\begin{split}
  \mathrm{(I)}&\geq -\frac{1}{2\delta_k} S(\mu_0^{\tau_k}-\nu_0^{\tau_k})-\frac{C}{2\delta_k} \big(|t_0^{\tau_k}-s_0^{\tau_k}|^2 +S(\mu_0^{\tau_k}-\nu_0^{\tau_k})\big)
\end{split}
\end{equation}for some constant $C>0$ independent of $k$.

Now we go to deal with term $\mathrm{(I\!I)}$. Similar to the estimate of \eqref{h-12}, there exists $C>0$ independent of $k$ such that
\begin{equation}\label{h-14}
\begin{split}
  \mathrm{(I\!I)}&=  \frac{1}{\delta_k}\sum_{m=1}^\infty \sum_{|\bm{n}|=m} c_{\bm n} \la \mu_0^{\tau_k}-\nu_0^{\tau_k},f_{\bm n}\raa \int_{\bar\oo} \mathrm{tr}(A\nabla^2 f_{\bm n}(x))\d(\nu_0^{\tau_k}-\mu_0^{\tau_k})(x)\\
  &\geq -\frac{1}{2\delta_k} \sum_{m=1}^\infty \sum_{|\bm{n}|=m} c_{\bm n} \Big[ \la \mu_0^{\tau_k}-\nu_0^{\tau_k},f_{\bm n}\raa^2 \!+\! \Big(\int_{\bar\oo}\!\mathrm{tr}\big(A \nabla^2 f_{\bm n}(x)\big) \d(\nu_0^{\tau_k}-\mu_0^{\tau_k})(x)\Big)^2\Big]\\
  &\geq -\frac{C}{\delta_k} \sum_{m=1}^\infty\sum_{|\bm{n}|=m} c_{\bm n} \la \mu_0^{\tau_k}-\nu_0^{\tau_k},f_{\bm n}\raa ^2=-\frac{C}{\delta_k} S(\mu_0^{\tau_k}-\nu_0^{\tau_k}).
\end{split}
\end{equation}

  At last, we estimate term $\mathrm{(I\!I\!I)}$. By virtue of \eqref{h-8.5}, $t_0^{\tau_k}$ converges to $\bar t_0$, $\mu_0^{\tau_k}$, $\nu_0^{\tau_k}$ converges weakly to $\bar\mu_0$ as $k\to \infty$. As $\oo$ is bounded, this also implies that $\W_2(\mu_0^{\tau_k},\bar\mu_0)\to 0$ as $k\to \infty$.  By $\mathrm{(H_3)}$,
  \begin{equation}\label{h-15}
  \begin{split}
   \lim_{k\to \infty } \mathrm{(I\!I\!I)}&=\lim_{k\to \infty}\Big\{\int_{\bar\oo}\!\big(\vartheta(s_0^{\tau_k}, x,\nu_0^{\tau_k}, \alpha)-\vartheta (t_0^{\tau_k},x, \mu_0^{\tau_k},\alpha)\big)\d \nu_0^{\tau_k}(x)\\
    &\qquad\qquad +\int_{\bar\oo}\big(\vartheta(t_0^{\tau_k},x,\mu_0^{\tau_k}, \alpha)-\vartheta(\bar t_0,x , \bar\mu_0,\alpha) \big)\d (\nu_0^{\tau_k}-\mu_0^{\tau_k})(x)\\
    &\qquad\qquad +\int_{\bar \oo} \vartheta(\bar t_0,x,\bar\mu_0,\alpha) \d(\nu_0^{\tau_k}-\mu_0^{\tau_k})(x)\Big\}\\
    &\geq \lim_{k\to \infty}\Big\{ - K_3\big(|s_0^{\tau_k}-t_0^{\tau_k}|+ \W_2(\nu_0^{\tau_k}, \mu_0^{\tau_k})\big) - 2K_3\big( |t_0^{\tau_k}-\bar t_0|+\W_2(\mu_0^{\tau_k},\bar \mu_0)\big)\\
    &\qquad\qquad +\int_{\bar\oo}\vartheta(\bar t_0,x,\bar\mu_0,\alpha)\d \nu_0^{\tau_k}(x)-\int_{\bar\oo}\vartheta(\bar t_0,x,\bar\mu_0,\alpha)\d \mu_0^{\tau_k}(x)\Big\}\\
    &=0.
  \end{split}
  \end{equation}

  Finally, inserting  the estimates \eqref{h-15}, \eqref{h-14}, \eqref{h-13} into \eqref{h-11}, due to \eqref{h-8.6}, we get
  \[\lim_{k\to \infty} -2\beta\geq \lim_{k\to \infty} \mathrm{(I)}+\mathrm{(I\!I)} +\mathrm{(I\!I\!I)}=0,\]
  which contradicts the fact $\beta>0$.

  Consequently, we have shown that the existence of $(\bar t,\bar\mu)$ satisfying \eqref{h-6} is false, and so is the existence of $(\tilde t,\tilde \mu) $. Thus, we conclude that $W(t,\mu)\leq V(t,\mu)$ for all $(t,\mu)\in\! [0,T)\!\times\!\pb(\bar\oo)$ as desired.
\end{proof}

Based on the comparison principle, Theorem \ref{thm-5}, the uniqueness of viscosity solution to HJB equation \eqref{e-1} with continuity property \eqref{a-4} can be proved in a standard way. Therefore, the value function $V(t,\mu)$ \eqref{a-2} associated with the optimal control problem for the reflected McKean-Vlasov SDE \eqref{a-1} can be  characterized as a unique viscosity solution to HJB equation \eqref{e-1}.

\end{document}